\documentclass[a4paper,reqno,11pt]{amsart}
\usepackage{amsmath,amstext,amssymb,amsopn,amsthm,mathrsfs}
\usepackage{xcolor}
\usepackage[matrix,arrow,ps]{xy}
 
\numberwithin{equation}{section}
\allowdisplaybreaks

\newtheorem{theorem}{Theorem}[section]
\newtheorem{proposition}[theorem]{Proposition}
\newtheorem{lemma}[theorem]{Lemma}
\newtheorem{corollary}[theorem]{Corollary}
\newtheorem{definition}[theorem]{Definition}

\newtheorem{example}[theorem]{Example}

\def\a{\alpha}
\def\b{\beta}

\def\vp{\varphi}
\def\ve{\varepsilon}

\newcommand{\N}{\mathbb{N}}

\newcommand{\R}{\mathbb{R}}
\newcommand{\C}{\mathbb{C}}
\newcommand{\Rd}{\mathbb{R}^d}
\newcommand{\bbM}{\mathbb{M}}

\newcommand*\calF{\mathcal{F}}
\newcommand*\calG{\mathcal{G}}
\newcommand*\calH{\mathcal{H}}

\newcommand*\calL{\mathcal{L}}

\newcommand*\calN{\mathcal{N}}

\newcommand*\calS{\mathcal{S}}

\newcommand{\bbG}{\mathbb{G}^\circ_{\a,\b}}

\newcommand{\D}{{\rm Dom}}
\newcommand{\esa}{essentially self-adjoint }
\newcommand{\calka}{\int_0^\infty}

\begin{document}
\title[A Grushin type operators]{Spectral analysis of Grushin type operators on the quarter plane}

\author[K. Stempak]{Krzysztof Stempak}
\address{Krzysztof Stempak \endgraf\vskip -0.1cm
         55-093 Kie\l{}cz\'ow, Poland \endgraf \vskip -0.1cm
				        }
\email{Krzysztof.Stempak@pwr.edu.pl}

\begin{abstract} 
We investigate  spectral properties of self-adjoint extensions of the operator 
$$
G_{\alpha,\beta}=-\Big(\frac{\partial^2}{\partial r^2}+\frac{2\a+1}{r}\frac{\partial}{\partial r} \Big)
-r^2 \Big(\frac{\partial^2}{\partial s^2}+\frac{2\b+1}{s}\frac{\partial}{\partial s} \Big),
$$
$\a,\b\in\R$, with domain $\D\, G_{\alpha,\beta}=C^\infty(\R^2_+)\subset L^2(\R^2_+,r^{2\a+1}s^{2\b+1}drds)$,  
which for some specific values of $\a,\b$, is a bi-radial part of the Grushin operator. Alternatively, we investigate $G^\circ_{\alpha,\beta}$, 
the Liouville form of $G_{\alpha,\beta}$, which is a symmetric and nonnegative operator on $L^2(\R^2_+, drds)$. 
One of the main tools used is an integral transform which combines the Laguerre scaled transform and the Hankel transform. 
Self-adjoint extensions $\mathbb{G}^\circ_{\alpha,\beta}$ of $G^\circ_{\alpha,\beta}$ are defined in terms of this transform, and the spectral decompositions 
of them are given. Another approach to construct self-adjoint extensions of $G^\circ_{\alpha,\beta}$, based on the technique of sesquilinear forms, 
is also presented and then the two approaches are compared. We also establish a closed form of the heat kernel corresponding to $\mathbb{G}^\circ_{\alpha,\beta}$.
\end{abstract}

\maketitle

\section{Introduction} \label{sec:intro} 

It is well known that the radial part of the (minus) Laplacian in $\Rd$ is the differential operator
$$
B=-\frac{d^2}{dr^2}-\frac{d-1}r\frac d{dr}
$$
in the sense that
$$
(-\Delta)F(x)=Bf(|x|), \qquad x\in\Rd,
$$
where $f=f(r)$, $r>0$, is the radial profile of a radial function $F$ on $\R^d$, i.e. $F(x)=f(|x|)$. Moreover, for radial functions, the Fourier 
transform reduces to the (modified) Hankel transform; this means that
$$ 
\calF_d F(x)=\calH_{(d-2)/2}f(|x|), \qquad x\in \Rd.
$$ 
Here,  for the  type parameter $\a>-1$, $\calH_\a$ is the integral transform given by
$$
\calH_\a f(r)=\int_0^\infty f(u)\frac{J_\a(ru)}{(ru)^\a}\,u^{2\a+1}\,du, \qquad r\in(0,\infty),
$$
where $J_\a$ stands for the Bessel function of the first kind of order $\a$, see e.g. \cite[Section 5]{Leb}.

The \textit{Grushin differential operator} on $\R^d=\R^{d_1}\times\R^{d_2}$, $d_1, d_2\ge1$, is given by the differential expression
$$
G=-\Delta_{x'}-|x'|^2\Delta_{x''},
$$
where $x=(x',x'')$, $x'\in\R^{d_1}$, $x''\in\R^{d_2}$,  and $\Delta_{x'},\Delta_{x''}$ are the Laplacians on $\R^{d_1}$ and $\R^{d_2}$, 
respectively. The symbol $|\cdot|$ stands for the Euclidean norm, e.g. in $\R^{d}$, $\R^{d_1}$, or $\R^{d_2}$, depending on the context. 
It is easily seen that initially considered with domain $C^\infty_c(\R^d)\subset L^2(\Rd)$, $G$ is symmetric and nonnegative. Moreover, 
$G$ is essentially self-adjoint, hence it admits the unique self-adjoint extension on $L^2(\Rd)$, called the \textit{Grushin operator}. 
During the last years numerous papers 
were devoted to different aspects of analysis of $G$. See, for instance, Dall'Ara and Martini \cite{DaM} and the extensive literature therein, 
and Jotsaroop et al. \cite{JST}, where, among others,  the spectral decomposition of the Grushin operator was proposed. Also some generalizations 
of $G$ were considered, see, e.g. Almeida et al. \cite{ABCS}.

We now modify the concept of radiality and call a function $F$ on $\R^d=\R^{d_1}\times\R^{d_2}$ \textit{bi-radial} provided $F(x)=f(|x'|,|x''|)$ 
for some $f\colon (0,\infty)\times(0,\infty)\to\C$. The structure of $G$ immediately reveals that for sufficiently smooth bi-radial functions we have
$$
G F(x',x'')=G_{\frac{d_1-2}2, \frac{d_2-2}2}f(r,s),\qquad r=|x'|,\quad s=|x''|,
$$
where for $\a,\b\in \R$ we set
$$
G_{\a,\b}=-\Big(\frac{\partial^2}{\partial r^2}+\frac{2\a+1}{r}\frac{\partial}{\partial r} \Big)
-r^2 \Big(\frac{\partial^2}{\partial s^2}+\frac{2\b+1}{s}\frac{\partial}{\partial s} \Big).
$$
Thus, for the specific values $\a=(d_1-2)/2$,  $\b=(d_2-2)/2$, this  differential operator can be seen as  the bi-radial part of the Grushin differential operator on 
$\R^d=\R^{d_1}\times\R^{d_2}$. 

The aim of this paper is to perform the spectral analysis of $G_{\a,\b}$, $\a,\b\in\R$, with domain $\D\,G_{\a,\b}=C^\infty_c\big(\R^2_+\big)\subset L^2(r^{2\a+1}s^{2\b+1}dr\,ds)$. In fact, we shall act simultaneously considering the unitarily equivalent version of $G_{\a,\b}$, see Section~\ref{ssec:Grushin2}, namely
\begin{equation*}
G_{\a,\b}^\circ = -\Big(\frac{\partial^2}{\partial r^2}-\frac{\a^2-1/4}{r^2}\Big)-r^2\Big(\frac{\partial^2}{\partial s^2}-\frac{\b^2-1/4}{s^2} \Big),
\end{equation*}
with domain $\D\,G_{\a,\b}^\circ=C^\infty_c\big(\R^2_+\big)$; see Section \ref{ssec:fra} for additional comments. This version, being a symmetric and 
nonnegative operator on $L^2\big(\R^2_+\big)$, is technically easier to work with because of the presence of Lebesgue measure. 
One of the main tools we use is an integral transform which combines the Laguerre scaled transform and the Hankel transform. The idea of using the 
Laguerre scaled transform is implicitly contained in the spectral decomposition of the Grushin operator presented in \cite{JST}. This decomposition 
is based on applying scaled Hermite functions, forming orthonormal bases in $L^2(\R^{d_1})$, together with the Fourier transform in $L^2(\R^{d_2})$. 
But, as it is well known, the radialization of the Hermite function setting leads to the Laguerre function framework; in addition, the radialization 
of the Fourier transform results in using the Hankel transform. Therefore, combining these two facts found an outcome in the idea of construction 
of the transform which is one of our main tools in the investigation of the Grushin-type operators. It is necessary to mention that the use of the 
systems of scaled Hermite functions was initiated and intensively applied by Thangavelu in the harmonic analysis of the sub-Laplacian on the Heisenberg 
group, \cite{T1}, \cite{T2}, and in an investigation of an analogue of Hardy's theorem for the Heisenberg group,  \cite{T3}, \cite{T4} (and also in 
some other contexts). See also \cite{T5}.

The paper is organized as follows. Section \ref{sec:prel} contains preliminaries and includes a discussion of the Bessel operator and the Hankel transform, 
the Laguerre operator and the Laguerre scaled transform, and the Grushin-type operators. In Section \ref{sec:Gru} we define the transform $\calG^\circ_{\a,\b}$, 
which combines the Laguerre scaled transform and the Hankel transform, and in Theorem~\ref{thm:first} we state its basic properties. Section~\ref{sec:self} 
is devoted to a construction of self-adjoint extensions of the Grushin-type operator $G^\circ_{\a,\b}$ given in terms of $\calG^\circ_{\a,\b}$. Then we 
discuss the corresponding heat semigroup and deliver a compact form of the associated heat kernels, see Theorem~\ref{thm:heat}. In Section~\ref{sec:self2}  another approach to construction of self-adjoint extensions is presented and then two approaches are compared. Here the main tools used are weak $\delta$-derivatives and $\delta$-Sobolev spaces, introduced and investigated in the $G^\circ_{\a,\b}$ context. Finally, in Section~\ref{sec:app} we gather comments, remainders and proofs of some technical results used earlier; we believe that putting all this stuff in the appendix will allow the reader to concentrate 
on the main line of thoughts. 

The final comment concerns the range of type parameters $\a,\b$, used in the statements of results that follow. For the $G^\circ_{\a,\b}$-setting it is
obvious that we can limit the full range  $\a,\b\in\R$ to $\a,\b\ge0$. The same holds for the $G_{\a,\b}$-setting, see Proposition~\ref{pro:first} and the
remarks following it. Despite this possible limitation we keep the full range whenever we can. But in several statements, for technical reasons we must 
limit the range to $\a,\b>-1$. This comes from the fact that the $\calG^\circ/\calG$-transforms are defined only for $\a,\b>-1$ (which is caused by earlier
range limitations for the Hankel and Laguerre transforms). It must be, however, emphasized that considering the range $\a,\b>-1$, larger than theoretically 
necessary  $\a,\b\ge0$, brings important advantages; one of them is exemplified by Theorem~\ref{thm:second}.

\textbf{Notation and terminology}. Throughout the paper we use fairly standard notions and symbols. An operator $T$ on a Hilbert space $(H,\langle\cdot,\cdot\rangle_H)$, is said to be \textit{nonnegative} provided $\langle Tf,f\rangle_H\ge0$ for all $f\in\D\,T\subset H$. 
The operators considered below 
with domains $C^\infty_c(0,\infty)$ or $C^\infty_c\big((0,\infty)\times (0,\infty)\big)$ in appropriate Hilbert spaces are densely defined; we shall not 
repeat this fact later on. We use the calligraphic letters $\calF$, $\calG$, $\calH$, $\calL$, with subscripts and/or superscripts, to denote 
transforms. For instance $\calF_d$ stands for the Fourier transform on $\Rd$ given by $\calF_d F(x)=(2\pi)^{-d/2}\int_{\Rd}F(y)\exp(-ix\cdot y)\,dy$, 
$F\in L^1(\Rd)$. Similarly, we use the \textsl{mathbb} font letters $\mathbb B$, $\mathbb D$, $\mathbb G$, $\mathbb L$, $\mathbb M$, with subscripts 
and/or superscripts, to denote self-adjoint operators. For $\a,\b\in\R$ we set 
$$
d\mu_\a(r)=r^{2\a+1}\,dr, \qquad  d\mu_{\a,\b}(r,s)=r^{2\a+1}s^{2\b+1}\,dr\,ds.
$$
The canonical inner products in the Hilbert spaces $L^2(d\mu_\a):=L^2\big((0,\infty),d\mu_\a\big)$ and $L^2(d\mu_{\a,\b}):=L^2(\R^2_+,d\mu_{\a,\b}\big)$, 
where $\R^2_+:=(0,\infty)\times(0,\infty)$, will be denoted by $\langle\cdot,\cdot\rangle_\a$, and $\langle\cdot,\cdot\rangle_{\a,\b}$, respectively. 
Relating two positive quantities $X$ and $Y$ we write $X\lesssim Y$, when $X\le cY$ for some $c>0$ independent of possible significant quantities entering $X$ and $Y$; the symbol $\gtrsim$ is understood analogously. If $X\lesssim Y$ and $Y\lesssim X$, then we write $X\simeq Y$. 

We mention that we shall slightly abuse the notation using the same symbols to denote differential operators (without precisely specified domains) and corresponding to them unbounded operators on appropriate $L^2$ spaces (with determined domains). We believe this will not lead to a confusion. 

\section{Preliminaries} \label{sec:prel}  

For the sake of clarity we decided to split this section into four subsections.
\subsection{Bessel operator} \label{ssec:Bessel}
The \textit{Bessel operator} given by the differential expression 
$$
B_\a= -\Big(\frac{d^2}{d r^2}+\frac{2\a+1}{r}\frac{d}{d r} \Big)=
\frac1{r^{2\a+1}}\Big(- \frac d{d r}\big(r^{2\a+1}\frac d{d r}\big)\Big),
$$
and considered with domain $\D\,B_\a=C^\infty_c(0,\infty)\subset L^2(d\mu_\a)$, is symmetric and nonnegative; the latter divergence form of $B_\a$ can be 
used to verify this. It is also known that $B_\a$ is essentially self-adjoint if and only if $|\a|\ge1$. This follows from \cite[Section 11]{E}. 
For $\a>-1$ the canonical self-adjoint extension of $B_\a$ is given in terms of the Hankel transform $\calH_\a$.~
\footnote{$\spadesuit$ In the literature $\calH_\a$ is sometimes called  the \textit{modified} Hankel transform, in opposition to $\calH_\a^\circ$, see Section 
\ref{ssec:Grushin2}, which is called the Hankel transform.} 
The kernel of this transform is built up from eigenfunctions of $B_\a$. More precisely, denoting
$$
\tilde{J}_{\a,\tau}(u)=J_\a(\tau u)/(\tau u)^\a, \qquad u>0,
$$
so that $B_\a \tilde{J}_{\a,\tau}=\tau^2 \tilde{J}_{\a,\tau}$ (this follows from the basic differential equation satisfied by the Bessel function $J_\a$), we have
$$
\calH_\a f(\tau)=\int_0^\infty f(u) \tilde{J}_{\a,\tau}(u)\,d\mu_\a(u), \qquad \tau>0.
$$
Specifying the type parameter to $\a=-1/2$ brings the cosine transform
$$
\calH_{-1/2} f(\tau)=\Big(\frac2{\pi}\Big)^{1/2}\int_0^\infty f(u)\cos \tau u\,du,
$$
and similarly, for $\a=1/2$ we obtain the sine transform.

Recall that $\calH_\a$, defined initially on $L^1(d\mu_\a)$, extends (uniquely) to a unitary isomorphism onto $L^2(d\mu_\a)$ and we  use the same 
symbol $\calH_\a$ to denote this extension. See  \cite[Section 4]{BS}, where Plancherel's identity and the inverse formula is proved for $\calH_\a$ 
in the full range $\a>-1$. We also add that $\calH_\a^{-1}=\calH_\a$ on $L^2(d\mu_\a)$. In the same way as the Fourier transform intertwines 
the Laplacian with the multiplication by $|\cdot|^2$, the Hankel transform $\calH_\a$ unitarily intertwines the Bessel operator $B_\a$ with the 
multiplication operator by $(\cdot)^2$, i.e.
\begin{equation}\label{2.1}
\calH_\a(B_\a f)= (\cdot)^2 \calH_\a f, \qquad f\in\D\, B_\a=C^\infty_c(0,\infty).
\end{equation}
Since the multiplication operator $\mathbb{M}_{(\cdot)^2}$ (with maximal domain) is self-adjoint and nonnegative on $L^2(d\mu_\a)$, see Lemma~\ref{lem:spe},  
the operator $\mathbb{B}_\a$ defined by 
\begin{align*}
\D\,\mathbb{B}_\a&:=\{f\in L^2(d\mu_\a)\colon (\cdot)^2 \calH_\a f\in L^2(d\mu_\a)\},\\
\mathbb{B}_\a&:=\calH_\a\circ \mathbb{M}_{(\cdot)^2}\circ \calH_\a
\end{align*}
is also self-adjoint and nonnegative. It may easily be checked that $\mathbb{B}_\a$ indeed extends $B_\a$. 

Finally, we mention that the Bessel differential operator $B_\a$ is homogeneous of degree 2, which means that $B_\a(\rho_\tau f)=\tau^2\rho_\tau(B_\a f)$,
where $\rho_\tau$, $\tau>0$, is the dilation $\rho_\tau f(r)= \tau^{\a+1}f(\tau r)$, $r>0$, being a unitary automorphism of $L^2(d\mu_\a)$. 
\subsection{Laguerre operator} \label{ssec:Laguerre}
We shall also make use of the \textit{Laguerre operator} with domain $\D\,L_\a=C^\infty_c(0,\infty)\subset L^2(d\mu_\a)$ defined by the differential expression 
$$
L_\a=B_\a+r^2= \frac1{r^{2\a+1}}\Big(- \frac d{d r}\big(r^{2\a+1}\frac d{d r}\big)+ r^{2\a+3}\Big).
$$
The latter form, the divergence form of $L_\a$, immediately shows that $L_\a$ is symmetric and nonnegative. 
It is also known that $L_\a$ is essentially self-adjoint if and only if $|\a|\ge1$; basically, this follows from \cite[Sections 27, 28]{E}, cf. also 
\cite[Section 7.2]{St2}. For $\a>-1$ the Laguerre functions (of convolution type) 
$$
\ell_n^\a(r)=c_{n,\a}L^\a_n(r^2)\exp(-r^2/2),\qquad  r>0, \quad n\in \N:=\{0,1,\ldots\}, 
$$ 
where $c_{n,\a}=\big(\frac{2\Gamma(n+1)}{\Gamma(n+\a+1)}\big)^{1/2}$, form an orthonormal basis in $L^2(d\mu_\a)$ and are eigenfunctions of $L_\a$, 
$$
L_\a \ell_n^\a=\lambda_{n}^{\a}\ell_n^\a, \qquad \lambda_{n}^{\a}=2(2n+\a+1), \quad n\in \N.
$$
Here $L^\a_n$ stands for the $n$th Laguerre polynomial of order $\a$, see e.g. Lebedev \cite[4.17]{Leb}. 

The mapping $f\to \{\langle f,\ell_n^\a\rangle_\a\}_{n=0}^\infty$, denote it $\calL_{\a,1}$, unitarily identifies $L^2(d\mu_\a)$ with $\ell^2(\N)$, 
and the inverse mapping $\calL_{\a,1}^{-1}$ is  $\{c_n\}_0^\infty\to \sum_0^\infty c_n \ell_n^\a$. Since the multiplication operator 
$\mathbb{M}_{\{\lambda_n^\a\}}$  is self-adjoint and nonnegative on $\ell^2(\N)$, see Lemma~\ref{lem:spe}, the operator 
$\mathbb{L}_\a$ given by 
\begin{align*}
\D\,\mathbb{L}_\a&:=\calL_{\a,1}^{-1}(\D\,\mathbb{M}_{\{\lambda_n^\a\}})=\{f\in L^2(d\mu_\a)\colon \sum_{n=0}^\infty|\lambda_{n}^{\a}\langle f,\ell^\a_n\rangle_\a|^2<\infty\},\\
\mathbb{L}_\a&:=\calL_{\a,1}^{-1}\circ \mathbb{M}_{\{\lambda_n^\a\}}\circ \calL_{\a,1}
\end{align*}
is self-adjoint and nonnegative on $L^2(d\mu_\a)$. 
A simple argument shows that indeed $\mathbb{L}_\a$ is an extension of $L_\a$.

In fact we shall need a \textit{scaled} Laguerre framework and by this we mean the following. 
For any $\tau>0$ the \textit{scaled system} 
$$
\ell_{n,\tau}^\a:=\rho_{\sqrt\tau}\ell_n^\a, \qquad n\in\N,
$$ 
is an orthonormal basis in $L^2(d\mu_\a)$, $\a>-1$. Moreover, the scaled system consists of eigenfunctions of the \textit{scaled Laguerre operator} 
$
L_{\a,\tau}=B_\a+\tau^2r^2,
$
with eigenvalues $\lambda_{n}^\a\tau$, i.e.,   
\begin{equation}\label{2.2}
L_{\a,\tau}\ell_{n,\tau}^\a=\lambda_{n}^\a\tau\,\ell_{n,\tau}^\a. 
\end{equation}
Indeed, 
to obtain \eqref{2.2} we write 
\begin{align*}
L_{\a,\tau}\ell_{n,\tau}^\a=(B_{\a}+\tau^2r^2 )( \rho_{\sqrt\tau}\ell_n^\a)=\tau  \rho_{\sqrt\tau}(B_\a\ell_n^\a)+\tau^2r^2 \rho_{\sqrt\tau}\ell_n^\a 
&=\tau  \rho_{\sqrt\tau}\big((B_\a+(\sqrt\tau r)^2  )\ell_n^\a\big)\\
&=\tau  \rho_{\sqrt\tau}\big(L_\a \ell_n^\a\big)\\
&=\tau\lambda_{n}^\a \rho_{\sqrt\tau}\ell_n^\a.
\end{align*}

In this scaled framework we consider the \textit{Laguerre scaled transform} which, for every $\tau>0$, attaches to a function $f\in L^2(d\mu_\a)$ 
the sequence of coefficients from the expansion of $f$ with respect to $\{\ell_{n,\tau}^\a\}_0^\infty$. Namely, for $n\in\N$,  $\tau>0$,
$$
\calL_{\a,\tau} f(n):=\calL_{\a} f(n,\tau)=\langle f,\ell_{n,\tau}^\a\rangle_\a, \qquad f\in L^2(d\mu_\a).
$$
By Parseval's identity, for every $\tau>0$, $\calL_{\a,\tau}\colon L^2(d\mu_\a)\to \ell^2(\N)$ is a unitary isomorphism, and its inverse 
$\calL_{\a,\tau}^{-1}\colon \ell^2(\N) \to L^2(d\mu_\a)$ is the mapping $\{c_n\}_0^\infty\to \sum_0^\infty c_n \ell_{n,\tau}^\a$. 
Notice also that as a counterpart to \eqref{2.1}, in the present framework we have for $n\in\N$,  $\tau>0$,
$$
\calL_{\a}(L_{\a,\tau}f)(n,\tau)=\lambda_{n}^\a\tau\, \calL_{\a}f(n,\tau), \qquad f\in \D\,L_{\a,\tau}=C^\infty_c(0,\infty).
$$
\begin{example}\label{ex1} We apply the formula (see, for instance, \cite[18.17.34]{NIST})
\begin{equation}\label{2.3}
\calka L^{\a}_n(u)e^{-bu} u^{\a}\,du=\frac{\Gamma(n+\a+1)}{\Gamma(n+1)}\frac{(b-1)^n}{b^{n+\a+1}}, \qquad \a>-1,\quad n\in\N, \quad b>0,
\end{equation} 
(for $b=1$ and $n=0$ the convention $0^0=1$ is used) to evaluate the Laguerre scaled  transform of the gaussian function $g(r):=\exp(-r^2/2)$ and to check, 
by a direct calculation, Parseval's identity for $g$, for every fixed $\tau>0$. We have 
$$
\calL_{\a} g(n,\tau)=\langle g,\ell_{n,\tau}^\a\rangle_\a=c_{n,\a}\tau^{(\a+1)/2}\calka e^{-r^2/2}L^\a_n(\tau r^2) e^{-\tau r^2/2}r^{2\a+1}dr. 
$$
A change of variable with the aid
 of \eqref{2.3} then leads to
$$
\calL_{\a} g(n,\tau)=\frac{2^{\a+1}}{c_{n,\a}}\Big(\frac{\sqrt{\tau}}{1+\tau}\Big)^{\a+1}\Big(\frac{1-\tau}{1+\tau}\Big)^n.
$$
Next, recalling that $c_{n,\a}=\big(2\Gamma(n+1)/\Gamma(n+\a+1)\big)^{1/2}$ and   applying
$$
\sum_{n=0}^\infty \frac{\Gamma(n+\a+1)}{\Gamma(\a+1)\Gamma(n+1)}q^n={\,}_2F_1(1,\a+1;1;q), \qquad \a>-1,\quad |q|<1,
$$
gives
$$
\sum_{n=0}^\infty |\calL_{\a} g(n,\tau)|^2=2^{2\a+1}\Gamma(\a+1)\Big(\frac{\sqrt{\tau}}{1+\tau}\Big)^{2(\a+1)}{\,}_2F_1\Big(1,\a+1;1;\Big(\frac{1-\tau}{1+\tau}\Big)^2\Big);
$$
here ${\,}_2F_1$ is the hypergeometric function. But
$$
{\,}_2F_1(1,\a+1;1;q)=\frac1{(1-q)^{\a+1}},
$$
and finally, after a calculation, 
$$
\sum_{n=0}^\infty |\calL_{\a} g(n,\tau)|^2=\frac12\Gamma(\a+1)=\int_0^\infty|g(r)|^2r^{2\a+1}dr.
$$
\end{example}

\subsection{Grushin-type operator, I} \label{ssec:Grushin}
With notation introduced in Section \ref{ssec:Bessel} we see that 
$$
G_{\a,\b}=B_{1,\a}+r^2 B_{2,\b},
$$
where $j=1,2$, now indicates the variable, $r$ (the 1\textit{st}) or $s$ (the 2\textit{nd}), over which the differentiation in the Bessel operators $B_\a$ and $B_\b$ 
is performed (more precisely, these operators are now assumed to be applied to functions of two variables). Recall that 
$\D\, G_{\a,\b}=C^\infty_c\big(\R^2_+\big)$. An easy calculation then shows that $G_{\a,\b}$ is symmetric and nonnegative 
in $L^2(d\mu_{\a,\b})$, hence it admits self-adjoint extensions. Notice also that the differential operator $G_{\a,\b}$ is \textit{parabolically homogeneous} of degree 2, 
in the sense that $G_{\a,\b}(\check\rho_\tau f)=\tau^2 \check\rho_\tau(G_{\a,\b}f)$, $\tau>0$, where the family of `parabolic' dilations 
$\{\check\rho_\tau\}_{\tau>0}$ now acts on functions on $\R^2_+$ by $\check\rho_\tau f(r,s)=f(\tau r,\tau^2 s)$. 

The appropriate scaling of the Laguerre functions is crucial in verification that the functions
$$  
\Psi_{n,\tau}(r,s):=\ell^\a_{n,\tau}(r)\tilde{J}_{\b,\tau}(s), \qquad n\in\N,\quad \tau>0,
$$
are eigenfunctions of the differential operator $G_{\a,\b}$. 
\begin{lemma} \label{lem:second}
Let $\a,\b>-1$. Then for $n\in\N$ and $\tau>0$, 
$$
G_{\a,\b}\Psi_{n,\tau}(r,s) = \lambda_{n}^\a\tau\, \Psi_{n,\tau}(r,s), \qquad r>0,\quad s>0.
$$
\end{lemma}
\begin{proof} 
Using $B_{\b} \tilde{J}_{\b,\tau}=\tau^2 \tilde{J}_{\b,\tau}$ and \eqref{2.2} gives
\begin{align*}
G_{\a,\b}\Psi_{n,\tau}(r,s) 
&=\big(B_{1,\a}+r^2 B_{2,\b} \big)\big(\ell^\a_{n,\tau}(r)\tilde{J}_{\b,\tau}(s)\big)\\
&=\big(B_{\a}\ell^\a_{n,\tau}\big)(r)\tilde{J}_{\b,\tau}(s)+r^2\ell^\a_{n,\tau}(r)(B_{\b}\tilde{J}_{\b,\tau})(s)\\
&=\big(B_{\a}\ell^\a_{n,\tau}\big)(r)\tilde{J}_{\b,\tau}(s)+\tau^2r^2\ell^\a_{n,\tau}(r)\tilde{J}_{\b,\tau}(s)\\
&=\big(L_{\a,\tau}\ell^\a_{n,\tau}\big)(r)\tilde{J}_{\b,\tau}(s)\\
&=\lambda_{n}^\a\tau\, \Psi_{n,\tau}(r,s).
\end{align*}
\end{proof}

For essential self-adjointness we have the following.
\begin{proposition} \label{pro:zero}
Let $\a,\b\in\R$.  If $|\a|\ge1$, then $G_{\a,\b}$ is essentially self-adjoint.
\end{proposition}
\begin{proof} 
In the proof we modify appropriately an argument used in Dall'Ara and Martini \cite[Section 3.2]{DaM}. 

Let $|\a|\ge1$ and assume that $\b>-1$; this restriction is justified, see the remarks following Proposition~\ref{pro:first}. By a basic criterion, 
see e.g. \cite[Proposition~3.8]{Sch}, it suffices to check that $\calN(G_{\a,\b}^*-\overline{\lambda} I)=\{0\}$ for $\lambda \in \C\setminus\R$; 
here $G_{\a,\b}^*$ stands for the adjoint of $G_{\a,\b}$. This amounts to verifying that the only $f\in L^2(d\mu_{\a,\b})$ that satisfies
$$
\forall \vp\in C^\infty_c(\R^2_+) \qquad \int_0^\infty\int_0^\infty f(r,s)\overline{(G_{\a,\b}-\lambda I)\vp(r,s)}\, d\mu_{\a,\b}(r,s)=0,
$$
must necessarily vanish a.e. Choosing $\vp(r,s)=\vp_1(r)\vp_2(s)$, $\vp_1,\vp_2\in C^\infty_c(0,\infty)$ we obtain 
\begin{align*}
0&=\int_0^\infty\Big(\int_0^\infty f(r,s)\overline{(B_{1,\a}+r^2B_{2,\b}-\lambda I)\vp_1(r)\vp_2(s)}\, d\mu_{\a}(r)\Big)d\mu_{\b}(s)\\
&=\int_0^\infty\Big(\int_0^\infty f(r,s)    \overline{B_\a \vp_1(r)}\,d\mu_{\a}(r)\Big) \overline{\vp_2(s)}\,d\mu_{\b}(s)\\
&\,\,\,+\int_0^\infty\Big(\int_0^\infty f(r,s) \overline{(r^2-\lambda)\vp_1(r)}\,d\mu_{\a}(r)\Big)  \overline{B_\b\vp_2(s)}\,d\mu_{\b}(s)\\
&=\int_0^\infty\Big(\int_0^\infty \calH_\b [f(r,\cdot)](\eta) \overline{(B_\a+\eta^2r^2-\eta^2\lambda)\vp_1(r)}\,d\mu_{\a}(r)\Big)  \overline{\calH_\b\vp_2(\eta)}\,d\mu_{\b}(\eta).
\end{align*}
For the last equality we used an equivalent version of Plancherel's identity for $\calH_\b$ and the fact that the Hankel transform $\calH_\b$ of the function
$$
s\to \int_0^\infty f(r,s)    \overline{B_\a \vp_1(r)}\,d\mu_{\a}(r),
$$ 
evaluated at $\eta>0$, with $r>0$ fixed, is $\int_0^\infty \calH_\b [f(r,\cdot)](\eta) \overline{B_\a\vp_1(r)}\,d\mu_{\a}(r)$, and similarly for 
$$
s\to \int_0^\infty f(r,s) \overline{(r^2-\lambda)\vp_1(r)}   \,d\mu_{\a}(r).
$$ 

Since  $\{\calH_\b\vp_2\colon \vp_2\in C^\infty_c(0,\infty)\}$ is dense in $L^2(d\mu_\b)$, we infer that for every $\vp_1\in C^\infty_c(0,\infty)$,
$$
\int_0^\infty \calH_\b [f(r,\cdot)](\eta) \overline{(B_\a+\eta^2r^2-\eta^2\lambda)\vp_1(r)}\,d\mu_{\a}(r)=0, \qquad \eta\in(0,\infty)-{\rm a.e}.
$$
But for every $\eta\in(0,\infty)$ the operator $L_{\a,\eta}=B_\a+\eta^2r^2$ with $\D\, L_{\a,\eta}=C^\infty_c(0,\infty)\subset L^2(d\mu_\a)$, is 
essentially self-adjoint,  see Lemma~\ref{lem:app1}, hence $\calH_\b [f(r,\cdot)](\eta)=0$ for a.e. $r\in(0,\infty)$. It follows that $f(r,\cdot)=0$ 
for a.e. $r\in(0,\infty)$ and finally $f=0$.
\end{proof}
It will follow from Theorem~\ref{thm:second} that for every $(\a,\b)\in (-1,1)\times(-1,1)\setminus\{(0,0)\}$, $G_{\a,\b}$ fails to be essentially 
self-adjoint, because for any such  $(\a,\b)$ at least two different self-adjoint extensions of $G_{\a,\b}$ are constructed. Currently, we are not 
able to verify if for $|\a|<1$ and $|\b|\ge1$, $G_{\a,\b}$ is or is not essentially self-adjoint and the same lack of answer concerns $G_{0,0}$.

\subsection{Liouville forms; Grushin-type operator, II} \label{ssec:Grushin2}
From now on we shall write $L^2(0,\infty)$ and $L^2\big(\R^2_+\big)$ to denote the spaces of square integrable functions on $(0,\infty)$ or $\R^2_+$, respectively, equipped with Lebesgue measures. Then $\langle\cdot,\cdot\rangle$ will stand for the canonical inner product (in both spaces; it will 
be clear from the context to which of the two spaces this symbol refers to). 

Although the term 'Liouville (normal) form' seems to be reserved for differential equations (see, for instance, \cite[Section 7]{E}), 
similarly to \cite[Section 6.2]{St2} we adopt it here in the operator theory context.~
\footnote{\textcolor[rgb]{1,0,0}{$\heartsuit$} It was not fortunate to consider the operator $\calL^\circ_{\{w,r,s\}}$, defined in \cite[Section 6.2]{St2}, as the `Liouville form' of 
$\calL_{\{w,r,s\}}$ for general Liouville triple $\{w,r,s\}$. However, in the case we consider, calling $B^{\circ}_\a$ and $L_\a^{\circ}$  the Liouville forms 
of $B_\a$ and $L_\a$ is adequate.} 
The Liouville forms of $B_\a$ and $L_\a$, $\a\in\R$, are
\begin{equation*}
B^{\circ}_\a=-\frac{d^2}{dr^2}+\frac{\a^2-1/4}{r^2},
\end{equation*}
and 
\begin{equation*}
L_\a^{\circ}=-\frac{d^2}{dr^2}+\frac{\a^2-1/4}{r^2}+r^2,
\end{equation*}
respectively; see \cite[Section 12]{E}. Both $B^{\circ}_\a$ and $L_\a^{\circ}$ are considered with domain $C^\infty_c(0,\infty)\subset L^2(0,\infty)$. 
Obviously, investigating $B^{\circ}_\a$ and $L_\a^{\circ}$ we could assume that $\a\ge0$ (notice that the Liouville forms of $B_{-\a}$ and $B_\a$, $\a\neq0$, coincide; the same holds for the pair $L_{-\a}$ and $L_\a$). For reasons already seen in the $B_\a$ and $L_\a$ contexts it is reasonable to consider $\a>-1$.

Let us mention that from now on we shall consequently apply the following convention: by affixing `$\circ$' as a superscript to an object previously
considered in the $L^2(d\mu_\a)$ or $L^2(d\mu_{\a,\b})$ setting, we shall mean a corresponding object related to its `Liouville form' in the $L^2(0,\infty)$ 
or $L^2\big(\R^2_+\big)$ setting, respectively.

The connection between $B_\a$ and $B^{\circ}_\a$ is expressed through the fact that these operators are intertwined by the unitary isomorphism 
$$
U_\a\colon L^2(d\mu_\a)\to L^2(0,\infty), \qquad U_\a f=(\cdot)^{\a+1/2}f.
$$
This means that $B^{\circ}_\a\circ U_\a=U_\a\circ B_\a$ on $\D\, B_\a$. Thus the spectral properties of $B_\a$ and $B^{\circ}_\a$ are the same. Notably, 
$B^{\circ}_\a$ is symmetric and nonnegative on $L^2(0,\infty)$, is essentially self-adjoint if and only if $\a\ge1$, and the canonical self-adjoint 
extension $\mathbb B_\a^\circ$ of $B_\a^\circ$ in $L^2(0,\infty)$, $\a\ge0$, is defined accordingly. Analogous comments concern the pair $L_\a$ and 
$L^{\circ}_\a$, and the relevant properties of the operator $L^{\circ}_\a$.

In particular, we can now replace $\calH_\a$ by $\calH_\a^\circ:=U_\a \circ \calH_\a\circ U_\a^{-1}$, $\a>-1$. More specifically, for suitable $f$, 
$$
\calH_\a^\circ f(\tau)=\int_0^\infty f(u) \tilde{J}_{\a,\tau}^\circ(u)\,du,  \qquad \tau>0,
$$
where $\tilde{J}_{\a,\tau}^\circ(u)=(\tau u)^{1/2}J_\a(\tau u)$ so that $B_\a^\circ \tilde{J}_{\a,\tau}^\circ=\tau^2 \tilde{J}_{\a,\tau}^\circ$. 
Clearly, $\calH_\a^\circ$  possesses properties analogous to $\calH_\a$, in particular, appropriate versions of Plancherel's identity and the inverse 
identity hold; also \eqref{2.1} holds with replacement of  $\calH_\a$ and $B_\a$ by $\calH_\a^\circ$ and $B_\a^\circ$, respectively.

In the new setting we next replace $\{\ell_n^\a\}_{n=0}^\infty$ by the system of Laguerre functions (of Hermite type) $\{\ell_n^{\a,\circ}\}_{n=0}^\infty$,~ 
\footnote{$\diamondsuit$ This is a nonstandard notation of Laguerre functions of Hermite type caused by our agreement for the use of the  `$\circ$' symbol. Usually, the symbol $\vp^\a_n$ is used to denote them.} 
$\a>-1$, 
$$
\ell_{n}^{\a,\circ}(r)=\Big(\frac{2\Gamma(n+1)}{\Gamma(n+\a+1)}\Big)^{1/2}L^\a_n(r^2)\exp(-r^2/2)r^{\a+1/2},\qquad r>0, 
$$
or, more generally, by the scaled system $\{\ell_{n,\tau}^{\a,\circ}\}_{n=0}^\infty$, $\tau>0$, where $\ell_{n,\tau}^{\a,\circ}:=U_\a \ell_{n,\tau}^\a$, i.e.
$$
\ell_{n,\tau}^{\a,\circ}(r)=\tau^{1/4}\ell_n^{\a,\circ}(\sqrt\tau r),\qquad r>0.
$$
Then $\{\ell_{n,\tau}^{\a,\circ}\}_{n=0}^\infty$ is an orthonormal basis in $L^2(0,\infty)$ consisting of  eigenfunctions of the scaled operator $L_{\a,\tau}^\circ:=B_\a^\circ+\tau^2 r^2$,
$$
L_{\a,\tau}^\circ \ell_{n,\tau}^{\a,\circ}=\lambda_{n}^{\a, \circ}\tau \ell_{n,\tau}^{\a,\circ}, \qquad \lambda_{n}^{\a, \circ}=\lambda_{n}^{\a}=2(2n+\a+1), \quad n\in \N.
$$
Further, we replace $\calL_\a$ by the transform $\calL^\circ_\a:=\calL_\a\circ U_\a^{-1}$, 
$$
\calL^\circ_{\a,\tau} f(n):=\calL^\circ_{\a} f(n,\tau)=\langle f,\ell^{\a,\circ}_{n,\tau}\rangle, \qquad f\in L^2(0,\infty), \quad n\in\N, \quad \tau>0,
$$ 
so that $\calL_{\a,\tau}^\circ\colon L^2(0,\infty)\to \ell^2(\N)$ is a unitary isomorphism, 
$\calL_{\a,\tau}^{\circ;-1}(\{c_n\})\to \sum_0^\infty c_n \ell^{\a,\circ}_{n,\tau}$, and 
$$
\calL_{\a}^\circ(L_{\a,\tau}^\circ \vp)(n,\tau)=\lambda_{n}^{\a,\circ}\tau \calL_{\a}^\circ \vp(n,\tau),\qquad \vp\in C^\infty_c(\R^2_+)\quad n\in\N,\quad\tau>0.
$$

We now call the operator 
\begin{equation*}
G_{\a,\b}^\circ=B_{1,\a}^\circ+r^2 B_{2,\b}^\circ
= \Big(-\frac{\partial^2}{\partial r^2}+\frac{\a^2-1/4}{r^2}\Big)+r^2\Big(-\frac{\partial^2}{\partial s^2}+\frac{\b^2-1/4}{s^2} \Big),
\end{equation*}
considered on $C^\infty_c\big(\R^2_+\big)$, the Liouville form of $G_{\a,\b}$; notice that 
$$
G_{1/2,1/2}^\circ=-\frac{\partial^2}{\partial r^2}-r^2\frac{\partial^2}{\partial s^2}=G_{-1/2,-1/2}. 
$$ 
Notice also that given $\a\ge0,\b\ge0$, the Liouville form of each  of the operators $G_{-\a,-\b}, G_{-\a,\b}$ and $G_{\a,-\b}$, 
coincides with $G_{\a,\b}^\circ$.

\section{$\calG^\circ$-transform and its inverse} \label{sec:Gru}  

In this section we define a transform which will be crucial for the spectral analysis of the Grushin-type operators. Using the scaled Laguerre transform 
in this definition was motivated by \cite{JST}, where the system of scaled Hermite functions was used to describe the spectral decomposition of the Grushin 
operator.

Let $\a,\b>-1$. For a suitable function $f=f(r,s)$ on $\R^2_+$, we define its transform $\calG_{\a,\b}^\circ f$ as a function on $\N\times(0,\infty)$ given 
as the integral transform based on the eigenfunctions of $G_{\a,\b}^\circ$, 
$$
\Psi_{n,\tau}^\circ(r,s):=\ell^{\a,\circ}_{n,\tau}(r)\tilde{J}_{\b,\tau}^\circ(s), \qquad n\in\N,\quad \tau>0,
$$
 Namely,
\begin{equation} \label{3.1}
\calG_{\a,\b}^\circ f(n,\tau)=\calka\calka f(r,s)\Psi_{n,\tau}^\circ(r,s)\,drds, \qquad n\in\N,\quad \tau>0,
\end{equation}
for any $f$ such that the double integral converges for every $n\in\N$ and almost every $\tau>0$. Obviously, $f\in C^\infty_c\big(\R^2_+\big)$ has this property.  

Notice, that  
$$
G_{\a,\b}^\circ\Psi_{n,\tau}^\circ=\lambda_{n}^{\a,\circ} \tau \Psi_{n,\tau}^\circ,
$$
so, using twice integration by parts, for $\vp\in C^\infty_c\big(\R^2_+\big)$ we obtain
\begin{equation}\label{3.2}
\calG_{\a,\b}^\circ (G_{\a,\b}^\circ \vp)(n,\tau)=\lambda_{n}^{\a,\circ}\tau \calG_{\a,\b}^\circ \vp(n,\tau), \qquad n\in\N,\quad \tau>0.
\end{equation}

The $\calG^\circ$-transform is for $\vp\in C^\infty_c\big(\R^2_+\big)$ expressible through the Hankel and Laguerre scaled transforms 
which we have to our disposal; recall that the Hankel transform $\calH_\b^\circ$, $\b>-1$, is a unitary automorphism of $L^2(0,\infty)$ and the Laguerre 
scaled  transform $\calL^\circ_{\a,\tau}$, $\a>-1$, $\tau>-1$, acts on $L^2(0,\infty)$ attaching to $g\in L^2(0,\infty)$ the  sequence  
$\{\langle g,\ell^\circ_{n,\tau}\rangle\}_{n=0}^\infty$. Namely, denoting $\vp_r:=\vp(r,\cdot)$ we have
\begin{equation}\label{3.3}
\calG_{\a,\b}^\circ \vp(n,\tau)=\calka \ell^{\a,\circ}_{n,\tau}(r)\calH_\b^\circ \vp_r(\tau)\,dr=\calL^\circ_{\a,\tau} \big[\calH_\b^\circ \vp_r(\tau)\big](n).
\end{equation}
The expression inside the square brackets is treated, for fixed $\tau$, as a function of $r$ to which $\calL^\circ_{\a,\tau}$ is applied. 
Similarly to this case, in several places below, square brackets will indicate that expressions inside them should be properly identified.
This means, that one has to recognize a variable among some other characters. We hope, this will not lead to a confusion, since from the 
context it will be relatively clear which characters appear as fixed parameters.


Observe that the performed calculations leading to \eqref{3.3} allow to extend the action of $\calG_{\a,\b}^\circ$ onto $L^2\big(\R^2_+\big)$ by setting
\begin{equation}\label{3.4a}
\calG_{\a,\b}^\circ f(n,\tau) = \calL^\circ_{\a,\tau} \big[\calH_\b^\circ f_r(\tau)\big](n), \qquad f\in L^2\big(\R^2_+\big).
\end{equation}
The correctness of this definition follows from 
$$
\calka\calka |\calH_\b^\circ f_r(\tau)|^2dr\,d\tau=\calka\calka |\calH_\b^\circ f_r(\tau)|^2d\tau\,dr=
\calka\calka |f_r(s)|^2ds\,dr=\|f\|^2_{L^2(\R^2_+)},
$$
so that $r\to \calH_\b^\circ f_r(\tau)$ is in $L^2(0,\infty)$ for a.e. $\tau\in(0,\infty)$, so that $\calL^\circ_{\a,\tau}$ can be applied. Note that if $f$ has separated variables, $f(r,s)=f_1(r)f_2(s)$, 
$f_i\in C^\infty_c(0,\infty)$, then we simply have
\begin{equation}\label{3.5}
\calG_{\a,\b}^\circ f(n,\tau) = \calL^\circ_{\a,\tau} f_1(n)\, \calH_\b^\circ f_2(\tau).
\end{equation}
From now on we assume the $\calG^\circ$-transform $\calG_{\a,\b}^\circ$ to be given by \eqref{3.4a}.

Let 
$$ 
L^2(\Gamma^\circ):=L^2(\N\times(0,\infty),dn\times d\tau),
$$ 
where  $dn$  is  the counting measure in $\N$ and $d\tau$ is Lebesgue measure in $(0,\infty)$, with norm
$$
\|F\|_{L^2(\Gamma^\circ)}=\Big(\sum_{n=0}^\infty \|F(n,\cdot)\|^2_{L^2(0,\infty)}\Big)^{1/2}=\Big(\sum_{n=0}^\infty \calka |F(n,\tau)|^2\,d\tau\Big)^{1/2}.
$$ 
We define the \textit{inverse transform} $\calG_{\a,\b}^{\circ;-1}$ on $L^2(\Gamma^\circ)$ by 
\begin{equation}\label{3.6}
\calG_{\a,\b}^{\circ;-1}F(r,s)=\calH_\b^\circ\big[\calL_{\a,\tau}^{\circ;-1}F_\tau (r)\big](s), \qquad r>0,\quad s>0.
\end{equation}
The above definition is correct because for $F\in L^2(\Gamma^\circ)$, for almost every $\tau$ we have $F_\tau:=F(\cdot,\tau)\in\ell^2(\N)$ and, for almost 
every $r>0$, the function $\tau\to \calL_{\a,\tau}^{\circ;-1}F_\tau (r)$ is in $L^2(0,\infty)$. Indeed, 
\begin{align*} 
\calka \calka |\calL_{\a,\tau}^{\circ;-1}F_\tau (r)|^2d\tau\,dr=\calka \calka |\calL_{\a,\tau}^{\circ;-1}F_\tau (r)|^2dr\,d\tau
&=\calka \sum_{n=0}^\infty |F(n,\tau)|^2 d\tau<\infty.
\end{align*}

\begin{theorem} \label{thm:first}
Let $\a,\b>-1$. For $f\in L^2\big(\R^2_+\big)$ we have Plancherel's identity
\begin{equation}\label{3.7}
\|\calG_{\a,\b}^\circ f\|_{L^2(\Gamma^\circ)}=\|f\|_{L^2(\R^2_+)}
\end{equation}
and the inverse formula
\begin{equation}\label{3.8}
f=\calG_{\a,\b}^{\circ;-1}\big(\calG_{\a,\b}^\circ f\big).  
\end{equation}
Moreover, for $F\in L^2(\Gamma^\circ)$ we have
\begin{equation}\label{3.9}
\|\calG_{\a,\b}^{\circ;-1} F\|_{L^2(\R^2_+)}=\|F\|_{L^2(\Gamma^\circ)}
\end{equation}
and
\begin{equation}\label{3.10}
F=\calG_{\a,\b}^{\circ}\big(\calG_{\a,\b}^{\circ;-1}F\big).  
\end{equation}
\end{theorem}
\begin{proof}
For \eqref{3.7} we first note that by Parseval's identity for  $\calL^\circ_{\a,\tau}$, $\tau>0$, we obtain 
$$
\sum_{n=0}^\infty |\calG_{\a,\b}^\circ f(n,\tau)|^2 =
\sum_{n=0}^\infty \big|\calL^\circ_{\a,\tau} \big[\calH_\b^\circ f_r(\tau)\big](n)\big|^2 =
\|\calH_\b^\circ f_r(\tau)\|^2_{L^2(0,\infty)}=\calka|\calH_\b^\circ(f_r)(\tau)|^2\,dr.
$$
Then integrating with $\tau$ and using Tonelli's theorem and Plancherel's identity for $\calH_\b^\circ $ gives
$$
\calka\Big(\sum_{n=0}^\infty|\calG_{\a,\b}^\circ f(n,\tau)|^2\Big)d\tau=\calka\calka|\calH_\b^\circ f_r(\tau)|^2\,d\tau\,dr=\calka\calka|f_r(s)|^2\,ds\,dr,
$$
which proves \eqref{3.7}.

To verify \eqref{3.8}, note that by \eqref{3.7}, indeed $\calG_{\a,\b}^\circ f\in L^2(\Gamma^\circ)$ so $\calG_{\a,\b}^{\circ;-1}$ can be applied. 
Hence, recalling that by \eqref{3.4a}, $(\calG_{\a,\b}^\circ f)_\tau=\calL^\circ_{\a,\tau} \big[r\to\calH_\b^\circ f_r(\tau)\big]$, using \eqref{3.6} 
gives for almost every $u>0$ 
\begin{align*} 
\calG_{\a,\b}^{\circ;-1}\big(\calG_{\a,\b}^\circ f\big)(u,\cdot)=\calH_\b^\circ\big[\tau\to\calL_{\a,\tau}^{\circ;-1}(\calG_{\a,\b}^\circ f)_\tau (u)\big]
&=\calH_\b^\circ\big[\tau\to\calL_{\a,\tau}^{\circ;-1}\big(\calL_{\a,\tau}^{\circ}[r\to\calH_\b^\circ f_r(\tau)]\big)(u)\big]\\
&=\calH_\b^\circ\big(\calH_\b^\circ f_u \big)\\ 
&= f_u;
\end{align*} 
to avoid a notational collision we wrote $(u,\cdot)$ rather than expected $(r,s)$, and to avoid getting lost at some places we 
prompted to which function $\calL^\circ_{\a,\tau}$ or $\calH_\b^\circ$ applies.

For \eqref{3.9} we write
\begin{align*}
\calka\calka|\calG_{\a,\b}^{\circ;-1}F(r,s)|^2dr\,ds&=\calka\calka \big|\calH_\b^\circ \big[\calL_{\a,\tau}^{\circ;-1}F_\tau (r)\big](s)\big|^2ds\,dr\\
&=\calka\calka\big|\calL_{\a,\tau}^{\circ;-1}F_\tau (r)\big|^2d\tau\,dr
=\|F\|^2_{L^2(\Gamma^\circ)}.
\end{align*}

To verify \eqref{3.10}, observing first that by \eqref{3.9} $\calG_{\a,\b}^{\circ;-1}F\in L^2(\R^2_+)$, we have 
$$
\calG_{\a,\b}^{\circ}\big(\calG_{\a,\b}^{\circ;-1}F\big)(n,\tau)
=\calL^\circ_{\a,\tau}\Big[\calH_\b^\circ\big((\calG_{\a,\b}^{\circ;-1}F)_r\big)(\tau)\Big](n)
=\big\langle  \calH_\b^\circ\big((\calG_{\a,\b}^{\circ;-1}F)_r\big)(\tau),\ell^{\a,\circ}_{n,\tau}\big\rangle.
$$
The first term in the last brackets is understood, with $\tau$ fixed, as a function of $r$ which is
$$
r\to \calH_\b^\circ\big((\calG_{\a,\b}^{\circ;-1}F)_r\big)(\tau)=\calH_\b^\circ\big(\calH_\b^\circ[\xi\to  \calL_{\a,\xi}^{\circ;-1}F_\xi(r)]\big)(\tau)=
\calL_{\a,\tau}^{\circ;-1}F_\tau(r)=\sum_{k=0}^\infty F(k,\tau)\ell^{\a,\circ}_{k,\tau}(r),
$$
hence the value of the last brackets, where $r$ is  also the variable of integration, indeed equals $F(n,\tau)$.
\end{proof}
\begin{corollary} \label{cor:first}
Let $\a,\b>-1$. The mappings
$$
\calG_{\a,\b}^{\circ}\colon  L^2\big(\R^2_+\big)\to L^2(\Gamma^\circ) \quad {\rm and} 
\quad \calG_{\a,\b}^{\circ;-1}\colon L^2(\Gamma^\circ)\to L^2\big(\R^2_+\big)
$$
are unitary isomorphisms.
\end{corollary}
\begin{proof}
It follows from \eqref{3.7} and \eqref{3.9} that $\calG_{\a,\b}^{\circ}$ and $\calG_{\a,\b}^{\circ;-1}$ are isometries, hence unitary mappings. 
On the other hand, \eqref{3.8} and \eqref{3.10} imply that $\calG_{\a,\b}^{\circ}$ and $\calG_{\a,\b}^{\circ;-1}$ are bijections, hence isomorphisms.
\end{proof}

We now reinterpret Example~\ref{ex1} evaluating the  $\calG_{\a,\b}^{\circ}$ transform of the skewed gaussian function 
$$
f(r,s)=r^{\a+1/2}s^{\b+1/2} \exp\big(-(r^2+s^2)/2\big), 
$$ 
and checking, by a direct calculation, Plancherel's identity for $f$.
\begin{example}\label{ex2} From Example \ref{ex1} it follows that for
 $g^\circ(r):=U_\a g(r)=r^{\a+1/2}\exp(-r^2/2)$ we have 
$$
\calL_{\a}^\circ g^\circ(n,\tau)=\frac{2^{\a+1}}{c_{n,\a}}\Big(\frac{\sqrt{\tau}}{1+\tau}\Big)^{\a+1}\Big(\frac{1-\tau}{1+\tau}\Big)^n.
$$
On the other hand, for $h^\circ(s):=s^{\b+1/2}\exp(-s^2/2)$ we have $\calH_\b^\circ h^\circ=h^\circ$ (recall that 
$\calH_\b^\circ = U_\b\circ \calH_\b\circ U_\b^{-1}$). Hence, for $f(r,s)=g^\circ(r)h^\circ(s)$ we obtain 
$$
\calG_{\a,\b}^\circ f(n,\tau)=\calL^\circ_\a (g^\circ)(n,\tau)\, \calH_\b^\circ(h^\circ)(\tau)=
\frac{2^{\a+1}}{c_{n,\a}}\Big(\frac{\sqrt{\tau}}{1+\tau}\Big)^{\a+1}\Big(\frac{1-\tau}{1+\tau}\Big)^n\tau^{\b+1/2}\exp(-\tau^2/2).
$$
It follows that
$$
\sum_{n=0}^\infty |\calG_{\a,\b}^\circ f(n,\tau)|^2= \tau^{2\b+1}\exp(-\tau^2) \sum_{n=0}^\infty |\calL_{\a}^\circ( g^\circ)(n,\tau)|^2= \tau^{2\b+1}\exp(-\tau^2) \frac12\Gamma(\a+1).
$$
Finally,
$$
\int_0^\infty\sum_{n=0}^\infty |\calG_{\a,\b}^\circ f(n,\tau)|^2\,d\tau =\frac14\Gamma(\a+1)\Gamma(\b+1)=\calka|f(r,s)|^2\,drds.
$$
\end{example}

\section{Self-adjoint extensions of $G^\circ_{\a,\b}$} \label{sec:self}  
For $\a>-1$ consider  the function 
$$
\Theta_\a\colon \N\times(0,\infty)\to(0,\infty),\qquad \Theta_\a(n,\tau)=\lambda^{\a,\circ}_n\tau,
$$
and the corresponding multiplication operator $\bbM^\circ_\a:=\bbM^\circ_{\Theta_\a}$ on $L^2(\Gamma^\circ)$, 
\begin{align*}
\D\,\bbM^\circ_\a&=\big\{F\in L^2(\Gamma^\circ)\colon \{\lambda^{\a,\circ}_n\tau F(n,\tau)\}_{(n,\tau)\in\N\times(0,\infty)}\in L^2(\Gamma^\circ)\big\},\\
\bbM^\circ_\a F(n,\tau)&=\lambda^{\a,\circ}_n\tau F(n,\tau)\quad {\rm for}\quad n\in\N,\,\tau>0, \qquad F\in\D\, \bbM^\circ_\a.
\end{align*}
Then $\bbM^\circ_\a$ is self-adjoint and nonnegative, and the spectrum of $\bbM^\circ_\a$ coincides with $[0,\infty)$; see Lemma~\ref{lem:spe}.
Observe also, that as a subspace of $L^2(\Gamma^\circ)$, $\D\,\bbM^\circ_\a$ does not depend on $\a>-1$. This 
is because the sequences $\{\lambda^{\a,\circ}_n\}_{n=0}^\infty$ are comparable for different $\a>-1$; we have
$$
\D\,\bbM^\circ_\a=\{F\in L^2(\Gamma^\circ)\colon \{(n+1)\tau F(n,\tau)\}_{(n,\tau)\in\N\times(0,\infty)}\in L^2(\Gamma^\circ)\big\}, \qquad \a>-1.
$$ 

Now, for  $\a,\b>-1$, we define $\bbG$, the transfer of $\bbM^\circ_\a$ on $L^2\big(\R^2_+\big)$ through the unitary isomorphisms 
$\calG_{\a,\b}^{\circ}\colon  L^2\big(\R^2_+\big)\to L^2(\Gamma^\circ)$ and its inverse $\calG_{\a,\b}^{\circ;-1}\colon L^2(\Gamma^\circ) \to L^2\big(\R^2_+\big)$, by 
\begin{align*}
\D\,\bbG&:=\calG_{\a,\b}^{\circ;-1}(\D\,\bbM^\circ_\a)= \{f\in L^2\big(\R^2_+\big)\colon \calG_{\a,\b}^{\circ}f\in \D\,\bbM^\circ_\a\},\\
\bbG&:=\calG_{\a,\b}^{\circ;-1}  \circ \bbM^\circ_\a\circ\calG_{\a,\b}^{\circ},
\end{align*}
It follows that $\bbG$ is self-adjoint and nonnegative, and the spectrum of $\bbG$ is $[0,\infty)$. 

The following theorem is the main result of this section. In its proof we use Proposition~\ref{pro:sec} whose justification requires relatively technical 
computations, hence it is postponed. 
\begin{theorem} \label{thm:second}
Let $\a,\b>-1$. Then each of the operators, $\bbG$, $\mathbb{G}^\circ_{-\a,\b}$ when $\a<1$, $\mathbb{G}^\circ_{\a,-\b}$ when $\b<1$, and 
$\mathbb{G}^\circ_{-\a,-\b}$ when $\a<1$ and $\b<1$, is a self-adjoint extension of $G^\circ_{\a,\b}$. Moreover, $\bbG$ and $\mathbb{G}^\circ_{-\a,\b}$ differ  when $0<\a<1$, $\bbG$ and $\mathbb{G}^\circ_{\a,-\b}$ differ  when $0<\b<1$, and $\bbG$ and $\mathbb{G}^\circ_{-\a,-\b}$ differ when $0<\a<1$ and $0<\b<1$.
\end{theorem}
\begin{proof} Fix $\a,\b>-1$. With the aid of \eqref{3.2} it follows (notice that $G^\circ_{\a,\b}$ maps $C^\infty_c\big(\R^2_+\big)$ into itself) that  
$C^\infty_c\big(\R^2_+\big)\subset \D\,\bbG$. For the main claim of the theorem it remains to show that 
\begin{equation}\label{4.1}
\bbG \vp=G^{\circ}_{\a,\b}\vp, \qquad \vp\in C^\infty_c\big(\R^2_+\big),
\end{equation}
because, as operators on $C^\infty_c\big(\R^2_+\big)$, $G^\circ_{\a,\b}$ remain unchanged when $\a$ is replaced by $-\a$, and/or $\b$ is replaced by $-\b$.
Now, to check \eqref{4.1} it suffices to verify that 
\begin{equation*}
\mathbb{M}_\a(\calG^{\circ}_{\a,\b} \vp)=\calG^{\circ}_{\a,\b}(G^{\circ}_{\a,\b}\vp), \qquad \vp\in C^\infty_c\big(\R^2_+\big),
\end{equation*}
which means checking that for every $n\in\N$ and $\tau>0$, we have
\begin{equation*}
\Theta_\a(n,\tau)\calG^{\circ}_{\a,\b}\vp(n,\tau)=\calka\calka G^{\circ}_{\a,\b}\vp(r,s)\Psi^\circ_{n,\tau}(r,s)\,dr\,ds.
\end{equation*}
But this holds because
\begin{align*}
\Theta_\a(n,\tau)\calka\calka \vp(r,s)\Psi^\circ_{n,\tau}(r,s)\,dr\,ds&=\calka\calka \vp(r,s)G^{\circ}_{\a,\b}\Psi^\circ_{n,\tau}(r,s)\,dr\,ds\\
&=\calka\calka G^{\circ}_{\a,\b}\vp(r,s)\Psi^\circ_{n,\tau}(r,s)\,dr\,ds, 
\end{align*}
where in the last step an integration by parts was used. Thus  \eqref{4.1} follows. 

The differences between the pairs of operators described in the last sentence  of the theorem are consequences of differences between the corresponding heat kernels; see Proposition~\ref{pro:sec}.
\end{proof}

The functional calculus for the self-adjoint operator $\bbG$, $\a,\b>-1$, will be the realization of the functional calculus for 
$\bbM^\circ_\a$ mapped by $\calG_{\a,\b}^{\circ}$ onto $L^2\big(\R^2_+\big)$. According to the remarks following Lemma~\ref{lem:spe}, in the particular 
case of $L^2(\Gamma^\circ)$ and $\bbM^\circ_\a$ as the multiplication operator by $\Theta_\a$, given a Borel function $\Phi\colon [0,\infty)\to \C$ we set
$\Phi(\bbM^\circ_{\Theta_\a})=M^\circ_{\Phi\circ\Theta_\a}$.

Consequently, the functional calculus for the self-adjoint operator $\bbG$ is the mapping 
$$
\Phi\to \Phi(\bbG):=\calG_{\a,\b}^{\circ;-1}\circ M^\circ_{\Phi\circ\Theta_\a}\circ\calG_{\a,\b}^{\circ},
$$
so that the domain of $\Phi(\bbG)$ is
$$
\D\,\Phi(\bbG):=\calG_{\a,\b}^{\circ;-1}(\D\,M^\circ_{\Phi\circ\Theta_\a } )  =\{f\in L^2\big(\R^2_+\big)\colon \calG_{\a,\b}^{\circ}f\in \D\,M^\circ_{\Phi\circ\Theta_\a }\}.
$$
Obviously, $\Phi(\bbG)$ inherits all the spectral properties of $\Phi(\bbM^\circ_{\Theta_\a})$. 

We now apply the above comments to the bounded functions $\Phi_t(y)=e^{-ty}$, $y\ge0$, where $t>0$ is a parameter. The corresponding family 
$\{\exp(-t \bbG)\}_{t>0}$ is the one-parameter semigroup of operators bounded on $L^2\big(\R^2_+\big)$, called the heat semigroup for $\bbG$. We prove that these operators are integral operators; the family of the corresponding kernels $\{K^\circ_{t;\a,\b}\}_{t>0}$ is then called the heat kernel for $\bbG$.

In what follows $I_\a$ stands for the modified Bessel function of the first kind of order $\a$, see \cite[Sec. 5.7]{Leb}. The statement and the proof 
of Proposition~\ref{pro:add}, which is an auxiliary technical result used in the proof of the following theorem, is postponed until the end of this section.
\begin{theorem} \label{thm:heat}
Let $\a,\b>-1$. For every $t>0$, $\exp(-t \bbG)$ is an integral operator with kernel 
$$
K^\circ_{t;\a,\b}\big((r,s),(u,v)\big)=\sqrt{rusv}\calka J_\b(\tau s)J_\b(\tau v)\exp\Big(-\frac12\frac{\tau(r^2+u^2)}{\tanh 2t\tau}\Big)I_\a\Big(\frac{\tau ru}{\sinh 2t\tau}\Big)\frac{\tau^2}{\sinh 2t\tau}\,d\tau,
$$
that is, for $f\in L^2(\R^2_+)$,
\begin{equation}\label{4.2}
\exp(-t \bbG)f(r,s)=\calka\calka K^\circ_{t;\a,\b}\big((r,s),(u,v)\big)f(u,v)\,du\,dv.
\end{equation}
\end{theorem}
\begin{proof}
Firstly, we notice that an easy computation with help of the asymptotics
\begin{equation}\label{4.3}
J_\b(y)\simeq y^\b \qquad {\rm and}\qquad  I_\a(y)\simeq y^\a, \qquad y\to 0^+,
\end{equation}
the bounds
\begin{equation}\label{4.4}
J_\b(y)=O(y^{-1/2})  \qquad {\rm and}\qquad  I_\a(y)=O(e^y y^{-1/2}), \qquad y\to \infty,
\end{equation}
and the limits $\lim_{y\to 0^+}\frac y{\sinh y}=\lim_{y\to 0^+}\frac y{\tanh y}=\lim_{y\to\infty}\tanh y=1$, shows that  
the integral representing $K^\circ_{t;\a,\b}\big((r,s),(u,v)\big)$ indeed converges for every $(r,s),(u,v)\in\R^2_+$. 
Secondly, notice that by Proposition~\ref{pro:add} and Schwarz` inequality, the integral in \eqref{4.2} converges for every $(r,s)\in\R^2_+$.

Using the comments preceding the statement of Theorem~\ref{thm:heat}, we obtain for $f\in L^2(\R^2_+)$, 
\begin{equation}\label{4.5} 
\exp(-t \bbG)f(r,s)=\calH^\circ_\b\Big[\sum_{n=0}^\infty e^{-t\tau\lambda^{\a,\circ}_{n}}\calG_{\a,\b}^{\circ}f(n,\tau)\ell^{\a,\circ}_{n,\tau}(r)\Big](s). 
\end{equation}
Note that by \eqref{3.7} $\calG_{\a,\b}^{\circ}f(n,\tau)\in L^2(\Gamma^\circ)$, hence
for almost every $r>0$, the function of $\tau$ inside the square brackets in \eqref{4.5}
is in $L^2(0,\infty)$ so that $\calH^\circ_\b$ can be applied.

From now on we temporarily assume that $f\in C^\infty_c(\R^2_+)$. For such $f$ the sum in  \eqref{4.5} becomes 
\begin{align*}
&\,\,\,\,\sum_{n=0}^\infty e^{-t\tau\lambda^{\a,\circ}_{n}}\calG_{\a,\b}^{\circ}f(n,\tau)\ell^{\a,\circ}_{n,\tau}(r)\\
&=\calka\calka f(u,v)\tilde{J}^\circ_{\b,\tau}(v)\Big(\sum_{n=0}^\infty e^{-t\tau\lambda^{\a,\circ}_{n}}\ell^{\a,\circ}_{n,\tau}(u)\ell^{\a,\circ}_{n,\tau}(r)  \Big)\,du\,dv\\
&=\calka\calka f(u,v)\tilde{J}^\circ_{\b,\tau}(v)\Big(\sum_{n=0}^\infty e^{-2t\tau(2n+\a+1)}\tau^{1/4}\ell^{\a,\circ}_{n}(\sqrt{\tau}u)\tau^{1/4}\ell^{\a,\circ}_{n}(\sqrt{\tau}r)\Big)\,du\,dv\\
&=\tau^{1/2} e^{-2t\tau(\a+1)}\calka\calka f(u,v)\tilde{J}^\circ_{\b,\tau}(v)\Big(\sum_{n=0}^\infty e^{-4t\tau n} \ell^{\a,\circ}_{n}(\sqrt{\tau}u)\ell^{\a,\circ}_{n}(\sqrt{\tau}r)\Big)\,du\,dv.\\
\end{align*}
Notice that exchanging summation with integration was possible since, assuming that ${\rm supp}\,f\subset[a,b]\times[a,b]$, $0<a<b<\infty$, and 
for $t,\tau,r>0$ fixed, we have
$$
\sum_{n=0}^\infty e^{-4t\tau n} \int_a^b \int_a^b |\ell^{\a,\circ}_{n}(\sqrt{\tau}u)\ell^{\a,\circ}_{n}(\sqrt{\tau}r)|\,du\,dv
\le \frac{b-a}{\sqrt{\tau}}\sum_{n=0}^\infty e^{-4t\tau n} |\ell^{\a,\circ}_{n}(\sqrt{\tau}r)|\|\ell^{\a,\circ}_{n}\|_{L^1(0,\infty)}
<\infty,
$$
and this is enough for our purposes; we used the fact that the norms $\|\ell^{\a,\circ}_{n}\|_{L^1(0,\infty)}$ and $\|\ell^{\a,\circ}_{n}\|_{L^\infty(0,\infty)}$ grow at most polynomially in $n\to\infty$; see e.g. \cite[(4.1)]{St0}.

The closed form of the sum in \eqref{4.5} is now obtained by the formula (see \cite[4.17.6]{Leb}),
$$
\sum_{n=0}^\infty \frac{\Gamma(n+1)}{\Gamma(n+\a+1)}L^\a_n(x)L^\a_n(y)w^n=\frac1{1-w} \exp\big(-\frac w{1-w}(x+y)\big)\frac1{(xyw)^{\a/2}}
I_\a\Big(\frac{2w^{1/2}}{1-w}(xy)^{1/2}\Big),
$$
valid for $|w|<1$, $\a>-1$.  Namely,
$$
\sum_{n=0}^\infty e^{-4t\tau n} \ell^{\a,\circ}_{n}(\sqrt{\tau}u)\ell^{\a,\circ}_{n}(\sqrt{\tau}r)= 
\frac{e^{2t\tau(\a+1)}}{\sinh 2t\tau} \exp\Big(-\frac12\frac{\tau(r^2+u^2)}{\tanh 2t\tau}\Big)(\tau ru)^{1/2}I_\a\Big(\frac{\tau ru}{\sinh 2t\tau}\Big),
$$
and consequently, 
\begin{align*}
&\sum_{n=0}^\infty e^{-t\tau\lambda^{\a,\circ}_{n}}\calG_{\a,\b}^{\circ}f(n,\tau)\ell^{\a,\circ}_{n,\tau}(r)\\
&=\frac{\tau^{1/2}}{\sinh 2t\tau}\calka\calka f(u,v)\tilde{J}^\circ_{\b,\tau}(v)\exp\Big(-\frac12\frac{\tau(r^2+u^2)}{\tanh 2t\tau}\Big)(\tau ru)^{1/2}
I_\a\Big(\frac{\tau ru}{\sinh 2t\tau}\Big)\,du\,dv.
\end{align*}
Treating the above expression as a function of $\tau$ and evaluating the Hankel transform $\calH^\circ_\b$ of it at $s$ finally shows that $\exp(-t \bbG)f(r,s)$ 
equals
\begin{align*}
&\calka\tilde{J}^\circ_{\b,\tau}(s) \frac{\tau^{1/2}}{\sinh 2t\tau}\calka\calka f(u,v)\tilde{J}^\circ_{\b,\tau}(v)
\exp\Big(-\frac12\frac{\tau(r^2+u^2)}{\tanh 2t\tau}\Big) \sqrt{\tau ur}I_\a\Big(\frac{\tau ru}{\sinh 2t\tau}\Big)\,du\,dv\,d\tau\\
&\,\,\,=\calka\calka K^\circ_{t;\a,\b}\big((r,s),(u,v)\big)f(u,v)\,du\,dv,
\end{align*} 
which proves \eqref{4.2} for $f\in C^\infty_c(\R^2_+)$ with $K^\circ_{t;\a,\b}$ as in the statement of Theorem~\ref{thm:heat}. 
Notice that exchanging the order of integration was possible since, assuming again  $f$ to be supported in $[a,b]\times[a,b]$, and considering 
for simplicity of notation the case $t=1/2$ only (the proof in the general case $t>0$ goes with minor changes), for $r,s>0$ fixed, we have
\begin{align*}
&\calka |J_\b(\tau s)| \frac{\tau^{2}}{\sinh\tau}\int_a^b \int_a^b (uv)^{1/2}|J_\b(\tau v)| \exp\Big(-\frac12\frac{\tau(r^2+u^2)}{\tanh\tau}\Big)
I_\a\Big(\frac{\tau ru}{\sinh\tau}\Big)\,du\,dv\,d\tau\\
&\lesssim \calka \tau|J_\b(\tau s)| \Big(\frac{\tau}{\sinh\tau}\Big)^{1+\a}\int_a^b  |J_\b(\tau v)|  \,dv\,d\tau<\infty,
\end{align*}
and this is enough for our purpose. We used the fact that $\frac{\tau}{\sinh\tau}<1$ for $\tau\in(0,\infty)$, then we applied 
\eqref{4.3} to bound $I_\a$ and cancel integration with respect to the $u$ variable. The concluding inequality easily follows by using \eqref{4.3} 
and \eqref{4.4} to bound $J_\b$.

For general $f\in L^2(\R^2_+)$ we now argue as follows. Let $\{f_n\}\subset C^\infty_c(\R^2_+)$ be such that $f_n\to f$ in $L^2(\R^2_+)$, hence also 
$\exp(-t \bbG)f_n\to \exp(-t \bbG)f$ in $L^2(\R^2_+)$. With $t>0$, $\a,\b$ fixed, define
$$
T^\circ g(r,s):=\calka\calka K^\circ_{t;\a,\b} \big((r,s),(u,v)\big)g(u,v)\,du\,dv, \qquad g\in L^2(\R^2_+), \quad (r,s)\in\R^2_+.
$$
By Proposition~\ref{pro:add} this definition is correct. By what we have already proved, combined with Proposition~\ref{pro:add} and Schwarz` inequality we obtain
$$
\exp(-t \bbG)f_n(r,s)=T^\circ f_n(r,s)\to T^\circ f(r,s), \qquad n\to\infty,
$$
for every $(r,s)\in\R^2_+$. It follows that $\exp(-t \bbG)f(r,s)=T^\circ f(r,s)$, $(r,s)$-a.e.
\end{proof}

Observe that $K^\circ_{t;\a,\b}\big((r,s),(u,v)\big)$ is continuous jointly in $t>0$, and $(r,s),(u,v)\in\R^2_+$ (this may be verified by elementary means),  symmetric in $(r,s)$ and $(u,v)$. Moreover, for $t>0$,
\begin{equation}\label{4.6} 
K^\circ_{t;\a,\b} \big((r,s),(u,v)\big)=t^{-3/2}K^\circ_{1;\a,\b} \Big(\big(\frac r{\sqrt t},\frac st\big),\big(\frac u{\sqrt t},\frac vt\big)\Big),
\quad (r,s),(u,v)\in\R^2_+.
\end{equation}
In addition, using
$$
J_{\pm 1/2}(y)=\Big(\frac2{\pi y}\Big)^{1/2}\left\{\begin{matrix}{\sin}\\ {\cos}\end{matrix}\right\}(y)\quad {\rm and}\quad
I_{\pm 1/2}(y)=\Big(\frac2{\pi y}\Big)^{1/2}\left\{\begin{matrix}{\sinh}\\ {\cosh}\end{matrix}\right\}(y),
$$
one can express $K^\circ_{t;\a,\b}\big((r,s),(u,v)\big)$ for $\a,\b\in\{-\frac12,\frac12\}$ in a simpler form. For instance, for $\a=\b=-\frac12$, 
\begin{align*}
&\,\,\,\,\,\,\,\,\,\,\,\,\,\,\,\,\,\,\,\,\,\,\,\,\,\,\,\,K^\circ_{t;-\frac12,-\frac12}\big((r,s),(u,v)\big)\\
&=\Big(\frac2{\pi}\Big)^{3/2}\calka \cos\tau s \cos\tau v
\exp\Big(-\frac12\frac{\tau(r^2+u^2)}{\tanh 2t\tau}\Big)
\sinh\Big(\frac{\tau ru}{\sinh 2t\tau}\Big)\Big(\frac{\tau}{\sinh 2t\tau}\Big)^{1/2}\,d\tau.
\end{align*}
Cleary, it would be interesting to know if $K^\circ_{t;\,\a,\b}$ is (\textit{strictly}) positive on $\R^2_+$ for arbitrary $\a,\b>-1$.

\begin{proposition} \label{pro:sec}
The heat kernels $\{K^\circ_{t;\a,\b}\}_{t>0}$ differ from each other for each of the following two pairs of type parameters taken from the set 
$(-1,1)\times(-1,1)$: i) $(\a,\b)$ and $(\a,-\b)$, $\b\neq0$; ii) $(\a,\b)$ and $(-\a,\b)$, $\a\neq0$; iii) $(\a,\b)$ and  $(-\a,-\b)$, $\a\neq0$ and $\b\neq0$.
\end{proposition}
\begin{proof} It suffices to show the difference for particular $t>0$ in each of the three considered cases; we choose $t=1/2$ for each of these cases. 
Then, on the diagonal, when $u=r$ and $v=s$, we obtain
$$
K^\circ_{1/2;\,\a,\b}\big((r,s),(r,s)\big)=rs\calka J_{\b}(\tau s)^2\exp\Big(-\frac{\tau r^2}{\tanh \tau}\Big)
I_\a\Big(\frac{\tau r^2}{\sinh \tau}\Big)\frac{\tau^2}{\sinh \tau}\,d\tau,
$$

In case i),  setting $r=1$ in the above integral and denoting
$$
F_{\a,\b}^{(1)}(s)= \calka J_{\b}(\tau s)^2\exp\Big(-\frac{\tau}{\tanh \tau}\Big)
I_\a\Big(\frac{\tau}{\sinh \tau}\Big)\frac{\tau^2}{\sinh \tau}\,d\tau, \qquad s>0, 
$$
we check that the functions $F_{\a,\b}^{(1)}$ and $F_{\a,-\b}^{(1)}$  differ. More precisely, we check that  $F_{\a,\pm\b}^{(1)}(s)$ have different 
behavior when $s\to0^+$. This will be a consequence of different asymptotics of $J_{\pm \b}$ at zero, see \eqref{4.3} (\eqref{4.3} and \eqref{4.4} 
are frequently used in this proof and we shall not recall them).   
First, we verify that 
$$ 
T(A):= \int_A^\infty \exp\Big(-\frac{\tau}{\tanh \tau}\Big)I_\a\Big(\frac{\tau}{\sinh \tau}\Big)\frac{\tau}{\sinh \tau}\,d\tau
$$
decays exponentially when $A\to\infty$. Indeed, for $\tau$ large, 
$$
I_\a\Big(\frac{\tau}{\sinh \tau}\Big)\lesssim \Big(\frac{\tau}{\sinh \tau}\Big)^\a\qquad {\rm and}\qquad \exp\Big(-\frac{\tau}{\tanh \tau}\Big)\le\exp(-\tau),
$$
hence
$$
T(A)\lesssim \int_A^\infty \exp(-\tau)\Big(\frac{\tau}{\sinh \tau}\Big)^{\a+1}\,d\tau\lesssim \int_A^\infty \exp(-\tau)\,d\tau\lesssim \exp(-A).
$$

Now, to fix the attention, assume $\b>0$ and estimate $F_{\a,\b}^{(1)}(s)$ on $0<s<1$ from above using $J_{\b}(\tau s)=O\big((\tau s)^\b\big)$ 
for $\tau<\frac 1s$, and  $J_{\b}(\tau s)=O((\tau s)^{-1/2})$ for $\tau\ge\frac 1s$,
\begin{align*}
F_{\a,\b}^{(1)}(s)&\lesssim s^{2\b}\int_0^{1/s} \frac{\tau^{2+2\b}}{\sinh \tau}\exp\Big(-\frac{\tau}{\tanh \tau}\Big)
I_\a\Big(\frac{\tau}{\sinh \tau}\Big) \,d\tau\\
&\,\,\,\,\,\,\,\, +\frac1s\int_{1/s}^\infty \exp\Big(-\frac{\tau}{\tanh \tau}\Big)I_\a\Big(\frac{\tau}{\sinh \tau}\Big) \frac{\tau}{\sinh \tau} \,d\tau\\
&\lesssim s^{2\b}+\frac1s\exp(-1/s)\lesssim s^{2\b}.
\end{align*}
On the other hand, using $J_{-\b}(\tau s)\simeq(\tau s)^{-\b}$ for $\tau<\frac 1s$, we estimate $F_{\a,-\b}^{(1)}(s)$ on $0<s<1$ from below, 
\begin{align*}
F_{\a,-\b}^{(1)}(s)&\gtrsim s^{-2\b}\int_0^{1/s}\frac{\tau^{2-\b}}{\sinh \tau}\exp\Big(-\frac{\tau}{\tanh \tau}\Big)
I_\a\Big(\frac{\tau}{\sinh \tau}\Big)\,d\tau\gtrsim s^{-2\b},
\end{align*}
simply by limiting the integration to the interval $(0,1)$.

In case ii), setting $s=1$ in the relevant integral, and denoting 
$$ 
F_{\a,\b}^{(2)}(r)= 
\calka J_{\b}(\tau)^2\exp\Big(-\frac{\tau r^2}{\tanh \tau}\Big)I_{\a}\Big(\frac{\tau r^2}{\sinh \tau}\Big)\frac{\tau^2}{\sinh \tau}\,d\tau, \qquad r>0,
$$ 
we verify that the functions $F_{\a,\b}^{(2)}$ and $F_{-\a,\b}^{(2)}$ differ by checking, similarly as in case i), that $F_{\pm\a,\b}^{(2)}(r)$ behave 
 differently when $r\to0^+$. This will be a consequence of different asymptotics of $I_{\pm \a}$ at zero. 

To fix the attention assume that $\a>0$. For $0<r<1$ and $0<\tau<\infty$ we have  $0<\frac{\tau r^2}{\sinh \tau}<1$. 
Hence, estimating first $I_\a$, and then $F_{\a,\b}^{(2)}(r)$ on $0<r<1$ from above gives
$$
F_{\a,\b}^{(2)}(r)\lesssim \int_0^\infty J_{\b}(\tau)^2\Big(\frac{\tau r^2}{\sinh \tau}\Big)^\a \frac{\tau^2}{\sinh \tau}\,d\tau=r^{2\a} 
\int_0^\infty J_{\b}(\tau)^2\Big(\frac{\tau}{\sinh \tau}\Big)^{1+\a}\tau\,d\tau\lesssim r^{2\a}.
$$
Similarly,  estimating first $I_\a$, then $J_\b$ and $F_{-\a,\b}^{(2)}(r)$ on $0<r<1$ from below gives
\begin{align*}
F_{-\a,\b}^{(2)}(r)\gtrsim r^{-2\a}\int_0^1 J_{\b}(\tau)^2 \exp\Big(-\frac{\tau r^2}{\tanh \tau}\Big)\Big(\frac{\tau}{\sinh \tau}\Big)^{1-\a}\tau\,d\tau
&\gtrsim r^{-2\a}\int_0^1 \tau^{2\b+1}\Big(\frac{\tau}{\sinh \tau}\Big)^{1-\a}\,d\tau\\
&\gtrsim  r^{-2\a}.
\end{align*}

Case iii) follows from case ii) since $F_{\a,\b}^{(2)}$ and $F_{-\a,-\b}^{(2)}$ differ.
\end{proof}
\begin{proposition} \label{pro:add}
Let $\a,\b>-1$. For every $t>0$ and $(r,s)\in\R^2_+$ we have
$$
K^\circ_{t;\a,\b}\big((r,s),(\cdot,\cdot)\big)\in L^2(\R^2_+).
$$
\end{proposition}
\begin{proof}
For simplicity of notation we consider the case $t=1/2$ only, the proof in the general case $t>0$ goes with minor changes. 
Now with $\a$ and $\b$ fixed, we write $K$ in place of $K^\circ_{1/2;\a,\b}$, and setting 
$$
F_{(r,s);u}(\tau):=\sqrt{rsu}J_\b(\tau s)\exp\Big(-\frac12\frac{\tau(r^2+u^2)}{\tanh\tau}\Big)I_\a\Big(\frac{\tau ur}{\sinh\tau}\Big)\frac{\tau^{3/2}}{\sinh\tau},
$$
we have 
$$
K\big((r,s),(u,v)\big)=\calka (\tau v)^{1/2}J_\b(\tau v)F_{(r,s);u}(\tau)\,d\tau=\calH_\b F_{(r,s);u}(v).
$$ 
Notice that $F_{(r,s);u}\in L^2(0,\infty)$. Then, by applying Plancherel`s identity for $\calH_\b$,
$$
\calka\calka|K\big((r,s),(u,v)\big)|^2\,dudv=\calka\calka|\calH_\b F_{(r,s);u}(v)|^2\,dvdu=\calka\calka| F_{(r,s);u}(\tau)|^2\,d\tau du.
$$
It remains to check that the last double integral is finite. Again, to simplify the notation we consider the case $r=s=1$ only; the proof for general $r,s>0$ 
is analogous. Then we split the inner integration onto $(0,1)$ and $(1,\infty)$, and use the bounds \eqref{4.3} and \eqref{4.4} for $J_\b$, to be left with the integrals
$$
A_0=\calka u\int_0^1 \tau^{2\b}\frac{\tau^{3}}{\sinh^2\tau}\exp\Big(-\frac{\tau(1+u^2)}{\tanh\tau}\Big)I_\a\Big(\frac{\tau u}{\sinh\tau}\Big)^2\,d\tau du
$$ 
and
$$
A_\infty=\calka u\int_1^\infty \frac 1{\tau}\frac{\tau^{3}}{\sinh^2\tau}\exp\Big(-\frac{\tau(1+u^2)}{\tanh\tau}\Big)I_\a\Big(\frac{\tau u}{\sinh\tau}\Big)^2\,d\tau du
$$ 
to bound. For $A_0$, using $\frac{\tau}{\sinh\tau}<1<\frac{\tau}{\tanh\tau}$ on $(0,1)$ and the asymptotic \eqref{4.3} for $I_\a$ gives 
$$
A_0\lesssim\calka u^{2\a+1}\exp(-u^2)\int_0^1 \tau^{2\b+1}\,d\tau du<\infty.
$$ 
To estimate $A_\infty$, using first $1<\frac{\tau}{\tanh\tau}$ on $(0,1)$, changing the order of integration and then splitting the inner integration  appropriately to apply the bound \eqref{4.4} on $I_\a$, we write
\begin{align*}
A_\infty &\lesssim\int_1^\infty \Big(\frac{\tau}{\sinh\tau} \Big)^2\calka u\exp(-u^2)I_\a\Big(\frac{\tau u}{\sinh\tau}\Big)^2\,du\, d\tau\\
&\lesssim\int_1^\infty \Big(\frac{\tau}{\sinh\tau} \Big)^{2\a+2}\int_0^{\frac{\sinh\tau}{\tau}}u^{2\a+1}\exp(-u^2)\,du\, d\tau+
  \int_1^\infty \Big(\frac{\tau}{\sinh\tau} \Big)^{\frac32}\int_{\frac{\sinh\tau}{\tau}}^\infty u^{\frac12}\exp(-u^2)\,du\, d\tau\\
&\lesssim\int_1^\infty \Big(\frac{\tau}{\sinh\tau} \Big)^{2(\a+1)}\, d\tau+
  \int_1^\infty \Big(\frac{\tau}{\sinh\tau}\Big)^{3/2}\,d\tau<\infty.
\end{align*}
\end{proof}

\section{Sesquilinear forms and self-adjoint extensions of $G^\circ_{\a,\b}$} \label{sec:self2}  
We begin with recalling a necessary part of a theory that links nonnegative sesquilinear forms and nonnegative self-adjoint operators on Hilbert spaces. 
For details,  see, e.g., \cite[Chapter 10]{Sch}, \cite{Ou}. 

In the general setting, for a (sesquilinear) form $\mathfrak{s}$ given on a Hilbert space $(H, \langle\cdot,\cdot\rangle_H)$ with dense domain 
$\D\,\mathfrak{s}$, the associated operator $A_\mathfrak{s}$ is defined by 
\begin{align*}
\D\,A_\mathfrak{s}&=\{h\in \D\,\mathfrak{s}\colon \exists u_h\in H\,\,\,\forall h'\in \D\,\mathfrak{s}\,\,\,\, 
\mathfrak{s}[h,h']=\langle u_h,h'\rangle_H\},\\
A_\mathfrak{s}h&=u_h \quad{\rm for}\quad h\in \D\,A_\mathfrak{s}.
\end{align*}
Then, if $\mathfrak{s}$ is Hermitian closed and nonnegative,  $A_\mathfrak{s}$ is self-adjoint and nonnegative. Recall that a nonnegative form $\mathfrak{s}$ 
is called closed provided the norm $\|h\|_{\mathfrak{s}}=\big(\mathfrak{s}[h,h]^2+\langle h,h\rangle_H^2  \big)^{1/2}$ defined on $\D\,\mathfrak{s}$ is complete. 

On the other hand, if $A$ is a  self-adjoint nonnegative operator on $H$, then the associated form is $\mathfrak{s}_A[x,y]:=\langle A^{1/2}x,A^{1/2}y\rangle_H$ 
with domain $\D\,\mathfrak{s}_A:=\D\,A^{1/2}$; $\mathfrak{s}_A$ is a densely defined Hermitian closed nonnegative form. 

It is an important observation in this theory that the mappings $\mathfrak{s}\mapsto A_\mathfrak{s}$ and $A\mapsto \mathfrak{s}_A$ 
are mutually reciprocal bijections between the set of all densely defined Hermitian closed nonnegative forms on $H$ and the set of all self-adjoint  
nonnegative operators on $H$. Notably, $\mathfrak{s}_{A_\mathfrak{s}}=\mathfrak{s}$ and $A_{\mathfrak{s}_A}=A$.

It is also well known that any densely defined symmetric and nonnegative (more generally, lower semibounded) operator $S$ has a self-adjoint extension 
which is also nonnegative (more generally, lower semibounded, which preserves the lower bound). The construction of this operator, denoted $S_F$ and 
nowadays called the Friedrichs extension of $S$, was given by Friedrichs in 1933. It goes as follows. Let $S$ be as above and let $\mathfrak{s}_S$ be 
defined as $\mathfrak{s}_S[x,y]=\langle Sx,y\rangle_H$ on $\D\,S$ so that $\mathfrak{s}_S$ is densely defined Hermitian and nonnegative. 
But more importantly, $\mathfrak{s}_S$ is closable, see \cite[Lemma~10.16]{Sch}. Let  $\overline{\mathfrak{s}_S}$ be the closure of  $\mathfrak{s}_S$. 
Although the completion procedure in the definition of $\overline{\mathfrak{s}_S}$ is abstract from its nature it can be shown that $\overline{\mathfrak{s}_S}$ may be realized in $H$, which means, in particular, that $\D\,\overline{\mathfrak{s}_S}\subset H$. Then, $S_F$ is just $A_{\overline{\mathfrak{s}_S}}$, 
the operator associated to $\overline{\mathfrak{s}_S}$. See, e.g.,  \cite[Definition 10.6]{Sch}. 

Coming back to the main line of our consideration, fix $\a,\b\in\R$ and define the operators
$$
\mathfrak{d}_1^\circ:=\mathfrak{d}_{1;\a}^\circ=\frac{\partial}{\partial r}-\frac{\a+1/2}r,\qquad \mathfrak{d}_1^{\circ,\dagger}
:=\mathfrak{d}_{1;\a}^{\circ,\dagger}=-\frac{\partial}{\partial r}-\frac{\a+1/2}r,
$$
and 
$$
\mathfrak{d}_2^\circ:=\mathfrak{d}_{2;\b}^\circ=r\Big(\frac{\partial}{\partial s}-\frac{\b+1/2}s\Big),\qquad 
\mathfrak{d}_2^{\circ,\dagger}:=\mathfrak{d}_{2;\b}^{\circ,\dagger}=r\Big(-\frac{\partial}{\partial s}-\frac{\b+1/2}s\Big),
$$
each with $C^\infty_c(\R^2_+)\subset L^2(\R^2_+)$ as domain. It is easily seen that $\mathfrak{d}_1^{\circ,\dagger}$ and $\mathfrak{d}_2^{\circ,\dagger}$, 
are the restrictions of the adjoint operators $\mathfrak{d}_1^{\circ,*}$ and $\mathfrak{d}_2^{\circ,*}$ to $C^\infty_c(\R^2_+)$. 
In particular, for $ j=1,2$, it holds
$$
\int_{\R^2_+} \mathfrak{d}_j^{\circ,\dagger}\vp(r,s) \overline{\psi(r,s)}\,drds
=\int_{\R^2_+}\vp(r,s) \overline{\mathfrak{d}_j^{\circ}\psi(r,s)}\,drds,  \qquad \vp,\psi\in C^\infty_c(\R^2_+).
$$
We call $\mathfrak{d}_j^\circ$, $j=1,2$, \textit{the partial delta-derivatives} associated to $G^\circ_{\a,\b}$.

Essential for further development is the decomposition
$$
G^\circ_{\a,\b}=\mathfrak{d}_1^{\circ,\dagger} \mathfrak{d}_1^\circ+\mathfrak{d}_2^{\circ,\dagger} \mathfrak{d}_2^\circ;
$$
notice that $\mathfrak{d}_j^\circ(C^\infty_c(\R^2_+))\subset C^\infty_c(\R^2_+)$, $j=1,2$. Observe that for fixed $\a,\b,$ the decomposition of 
$G^\circ_{\a,\b}$ of this type is not unique, since the choice of $\mathfrak{d}_{1;\pm\a}^\circ$ and $\mathfrak{d}_{2;\pm\b}^\circ$, 
with the signs $+/-$ taken independently, brings the same outcome. 

Connected to the partial delta-derivatives is the concept of the weak partial derivatives.
\begin{definition}\label{def:weak}
Let $\a,\b\in\R$ be given, $f\in L^1_{\rm loc}(\R^2_+)$, and $j\in\{1,2\}$. We say that the weak (partial) $\mathfrak{d}_j^\circ$-\textit{derivative} of $f$ exists provided there is $g_j\in L^1_{\rm loc}(\R^2_+)$ such that
$$
\int_{\R^2_+}\mathfrak{d}_j^{\circ,\dagger}\vp(r,s)\overline{f(r,s)}\,drds=\int_{\R^2_+} \vp(r,s)\overline{g_j(r,s)}\,drds, \qquad \vp\in C^\infty_c({\R^2_+}).
$$ 
Then we set $\mathfrak{d}_{j,{\rm weak}}^{\circ}f:=g_j$  and call $g_j$ the weak $\mathfrak{d}_j^\circ$-derivative of $f$.
\end{definition}

Clearly, for $\a=-1/2$, the notion of weak $\mathfrak{d}_1^\circ$-derivative coincides with the notion of the usual  weak partial derivative. 
We write $\partial_{j,{\rm weak}}f$,  to denote the usual weak partial derivatives of $f$ on $\R^2_+$. For general $\a$ we have the following.

\begin{lemma} \label{lem:weak}
Let $\a,\b\in\R$ and $f\in L^1_{\rm loc}(\R^2_+)$. Then, for $j=1,2$, $\mathfrak{d}_{j,{\rm weak}}^\circ f$ exists if and only if $\partial_{j,{\rm weak}}f$ exists; moreover
$$
\mathfrak{d}_{1,{\rm weak}}^\circ f=\partial_{1,{\rm weak}}f-\frac{\a+1/2}r f\quad and \quad 
\mathfrak{d}_{2,{\rm weak}}^\circ f=r\Big(\partial_{2,{\rm weak}}f-\frac{\b+1/2}r f\Big).
$$ 
\end{lemma}
\begin{proof}
The proof  is straightforward. Let $\a,\b\in\R$  and $f\in L^1_{\rm loc}(\R^2_+)$ be given. For $j=1$ we verify that the conditions
\begin{equation}\label{con1}
\exists g_1\in L^1_{\rm loc}(\R^2_+)\,\,\,\forall \vp\in C^\infty_c(\R^2_+) \,\,\,\quad\int \mathfrak{d}_1^{\circ,\dagger}\vp\,\overline{f}=\int\vp\,\overline{ g_1}
\end{equation}
(here and below, integration goes over $\R^2_+$ with respect to $drds$) and 
\begin{equation}\label{con2}
\exists h_1\in L^1_{\rm loc}(\R^2_+)\,\,\,\forall \vp\in C^\infty_c(\R^2_+) \,\,\,\quad\int \partial_1\vp\,\overline{f}=-\int\vp\,\overline{h_1},
\end{equation}
are equivalent and, moreover, $g_1=h_1-\frac{\a+1/2}r f$. Indeed, \eqref{con1} asserts  that there is $g_1$ such that
$$
\int \partial_1\vp\,\overline{f}=-\int\vp\,\overline{ g_1}-\int\vp\,\frac{\a+1/2}r \overline{f}=-\int\vp\,\overline{\big(g_1-\frac{\a+1/2}r f\big)}, 
$$
so that  \eqref{con2} follows with $h_1=g_1+\frac{\a+1/2}r f$. Conversely,  \eqref{con2} leads to 
$$
\int \mathfrak{d}_1^{\circ,\dagger}\vp\,\overline{f}=\int\vp\,\overline{h_1}-\int\vp\,\frac{\a+1/2}r \overline{f}=\int\vp\,\overline{\big(h_1-\frac{\a+1/2}r f\big)},
$$
so that  \eqref{con1} follows with $g_1=h_1-\frac{\a+1/2}r f$. The case $j=2$ is treated completely analogously.
\end{proof}

Despite of the above result, which equalizes the existence of $\partial_{j,{\rm weak}}f$ with that of $\mathfrak{d}_{j,{\rm weak}}^\circ f$, the concept 
of $\mathfrak{d}_j^\circ$-derivatives is still convenient in the following definition of Sobolev-type spaces (of first order) connected to $G^\circ_{\a,\b}$. 
\begin{definition}\label{def:sob}
Let $\a,\b\in\R$. The delta-Sobolev space $W_{\a,\b}^\circ(\R^2_+)$ is the space
$$
W_{\a,\b}^\circ(\R^2_+)=\big\{f\in L^2(\R^2_+)\colon \mathfrak{d}_{j,{\rm weak}}^\circ f \,\, exist \,\,for\,\, j=1,2,\,\, and  \,\,are \,\,in\,\,\,
L^2(\R^2_+)\big\},
$$ 
equipped with the inner product
$$
\langle f,g\rangle_{W_{\a,\b}^\circ(\R^2_+)}=\langle f,g\rangle +\sum_{j=1}^2
\langle \mathfrak{d}_{j,{\rm weak}}^\circ f,\mathfrak{d}_{j,{\rm weak}}^\circ g\rangle .
$$
The closure of $C^\infty_c(\R^2_+)$ in $W_{\a,\b}^\circ(\R^2_+)$ with respect to the norm $\|\cdot\|_{W_{\a,\b}^\circ  (\R^2_+)}$ generated by 
$\langle\cdot,\cdot\rangle_{W_{\a,\b}^\circ(\R^2_+)}$ is denoted $W_{0;\a,\b}^\circ(\R^2_+)$.
\end{definition}

The following proposition has a relatively standard proof hence we omit it.

\begin{proposition} \label{pro:Sob1}
The delta-Sobolev space $W_{\a,\b}^\circ(\R^2_+)$ is a Hilbert space.
\end{proposition}

Let $\a,\b\in\R$. Define the form $\mathfrak{t}^\circ=\mathfrak{t}^\circ_{\a,\b}$ with domain $W_{\a,\b}^\circ(\R^2_+)$ by
$$
\mathfrak{t}^\circ[f,g]=\int_{\R^2_+}\sum_{j=1}^2  \mathfrak{d}_{j,{\rm weak}}^\circ f(r,s)\overline{\mathfrak{d}_{j,{\rm weak}}^\circ g(r,s)}\,drds. 
$$
The form $\mathfrak{t}^\circ$ restricted to  $W_{0;\a,\b}^\circ(\R^2_+)$ is denoted $\mathfrak{t}^\circ_0$. 
Thus $\D\,\mathfrak{t}^\circ=W_{\a,\b}^\circ(\R^2_+)$ and $\D\,\mathfrak{t}^\circ_0=W_{0;\a,\b}^\circ(\R^2_+)$. Obviously, $\mathfrak{t}^\circ$ and 
$\mathfrak{t}^\circ_0$ are  Hermitian nonnegative and densely defined on $L^2(\R^2_+)$.

In what follows, $\mathbb{D}_{\a,\b}^\circ$ and $\mathbb{D}_{0;\a,\b}^\circ$ denote the operators associated with forms $\mathfrak{t}^\circ$ and 
$\mathfrak{t}^\circ_0$, respectively.

\begin{theorem} \label{thm:oper}
Let $\a,\b>-1$. The operators $\mathbb{D}_{\a,\b}^\circ$ and $\mathbb{D}_{0;\a,\b}^\circ$ are self-adjoint and nonnegative extensions of $G^\circ_{\a,\b}$. 
Moreover, $\mathbb{D}_{0;\a,\b}^\circ$ is the Friedrichs extension of $G^\circ_{\a,\b}$.
\end{theorem}
\begin{proof} 
To complete the list of properties of $\mathfrak{t}^\circ$ we observe that $\mathfrak{t}^\circ$  is closed; this fact is just  completeness of the norm 
$\|\cdot\|_{W_{\a,\b}^\circ(\R^2_+)}$. Consequently, $\mathfrak{t}^\circ_0$ is closed. Thus, by the general theory, $\mathbb{D}_{\a,\b}^\circ$ and 
$\mathbb{D}_{0;\a,\b}^\circ$ are self-adjoint and nonnegative. We now check that they extend $G^\circ_{\a,\b}$. From the general definition we have 
$$
\D\,\mathbb{D}_{0;\a,\b}^\circ=\{f\in W_{0;\a,\b}^\circ(\R^2_+)\colon \exists u_f\in L^2(\R^2_+)\,\,\,\forall g\in W_{0;\a,\b}^\circ(\R^2_+)\,\,\, 
\mathfrak{t}^{\circ}_0[f,g]=\langle u_f,g\rangle \},
$$
and $\mathbb{D}_{0;\a,\b}^\circ f=u_f$ for $f\in \D\,\mathbb{D}_{0;\a,\b}^\circ$, and similarly for $\mathbb{D}_{\a,\b}^\circ$. 

We claim that $C^\infty_c(\R^2_+)\subset \D\,\mathbb{D}_{0;\a,\b}^\circ$ and $\mathbb{D}_{0;\a,\b}^\circ\vp=G^\circ_{\a,\b}\vp$ for 
$\vp\in C^\infty_c(\R^2_+)$, and analogously for $\mathbb{D}_{\a,\b}^\circ$. For this purpose it suffices to check (in both cases) that given 
$\vp\in C^\infty_c(\R^2_+)$, for every $g\in W_{\a,\b}^\circ(\R^2_+)$ it holds
\begin{equation}\label{part}
\int_{\R^2_+} \sum_{j=1}^2\mathfrak{d}_{j}^\circ \vp\, \overline{\mathfrak{d}_{j,{\rm weak}}^\circ g}
=\langle G^\circ_{\a,\b} \vp,g\rangle 
\end{equation}
(integration with respect to $drds$ here and below). But expanding both sides of \eqref{part} with the aid of Lemma~\ref{lem:weak}, canceling common terms and abbreviating $\a^*:=\a+1/2$, $\b^*:=\b+1/2$, $\partial_jg:=\partial_{j,{\rm weak}}g$ leads to 
the equivalent version of \eqref{part}, namely
\begin{align}\label{part2}
\int_{\R^2_+}\big(\partial_1\vp-\frac{\a^*}{r} \vp\big)\overline{\partial_1g}+&\int_{\R^2_+}r^2\big(\partial_2\vp-\frac{\b^*}r\vp\big)\overline{\partial_2g}\nonumber \\ 
&=-\int_{\R^2_+}\Big(\partial_1\big(\partial_1\vp-\frac{\a^*}{r} \vp\big) + \partial_2\big(r^2\big(\partial_2\vp-\frac{\b^*}r \vp\big)\big)\Big)\overline{g}.
\end{align}

Now, let $\Omega\subset \R^2_+$ be a bounded open subset containing ${\rm supp}\,\vp$, separated from the boundary of $\R^2_+$, with $\partial \Omega$ being 
sufficiently smooth. 
Equivalently, we prove \eqref{part2} with integration over $\Omega$ replacing $\R^2_+$. For this, abbreviating further
$\psi_1:=\partial_1\vp-\frac{\a^*}{r} \vp$,  $\psi_2:=r^2(\partial_2\vp-\frac{\b^*}{r} \vp)$, so that $\psi_j\in C^\infty_c(\R^2_+)$, it suffices to check that  
\begin{equation}\label{kkk2}
\int_\Omega \psi_1 \overline{\partial_1g} = -\int_\Omega \partial_1\psi_1 \,\overline{g}\quad {\rm and}\quad  
\int_\Omega \psi_2\, \overline{\partial_2g} = -\int_\Omega \partial_2\psi_2 \, \overline{g}.
\end{equation}
But $g\in W_{\a,\b}^\circ(\R^2_+)$ hence, with the aid of Lemma~\ref{lem:weak}, it follows that  $g\in H^1(\Omega)$, where the  $H^1(\Omega)$ 
stands for the Euclidean Sobolev space. Thus, the assumptions needed in Gauss' formula, cf. \cite[D.4,\,p.\,408]{Sch}, are satisfied, hence \eqref{kkk2} holds true. 
This finishes verification of the claim and hence the proof of 
the main part of the theorem. 

It remains to prove that $\mathbb{D}_{0;\a,\b}^\circ=\big(G^\circ_{\a,\b}\big)_F$. We take the form 
$\mathfrak{s}^\circ[f,g]=\langle G^\circ_{\a,\b}f,g\rangle $ 
on the domain $\D\, \mathfrak{s}^\circ=C^\infty_c(\R^2_+)$ and consider its closure $\overline{\mathfrak{s}^\circ}$. We claim that 
\begin{equation}\label{cl}
\overline{\mathfrak{s}^\circ}=\mathfrak{t}_0^\circ.
\end{equation}
This is enough for our purpose since then, with the previous notation, we have
$$
\D\,(G^\circ_{\a,\b})_F=\D\,A_{\overline{\mathfrak{s}^\circ}}=\D\,A_{\mathfrak{t}^\circ_0},
$$
and as one immediately sees, the latter space coincides with $\D\,\mathbb{D}^{\circ}_{0;\a,\b}$. Moreover, it follows that 
$(G^\circ_{\a,\b})_Ff=\mathbb{D}^{\circ}_{0;\a,\b}f$ for $f$ from these joint domains. Returning to \eqref{cl}, we note that it is a consequence of the fact 
that $C^\infty_c(\R^2_+)$ lies densely in $\D\,\mathfrak{t}^{\circ}_0=W_{0;\a,\b}^\circ(\R^2_+)$ and $\mathfrak{t}^{\circ}_0$ is closed. Here are details.  Clearly $\mathfrak{t}^{\circ}_0$ extends $\mathfrak{s}^\circ$ and hence the inclusion $\subset$ follows. For the opposite inclusion let 
$f\in \D\,\mathfrak{t}^{\circ}_0=W_{0;\,\a,\b}^\circ(\R^2_+)$ and take $\{\vp_n\}\subset C^\infty_c(\R^2_+)$ such that $\vp_n\to f$ in $W_{0;\,\a,\b}^\circ(\R^2_+)$. In particular, this means that $\vp_n\to f$ and $\mathfrak{d}^\circ_j \vp_n\to \mathfrak{d}^\circ_{j,{\rm weak}}f$ in $L^2(\R^2_+)$, $j=1,2$. 
We want to show that $f\in \D\,\overline{\mathfrak{s}^\circ}$.

But for this purpose it suffices (see \cite[p.224]{Sch}) to ensure existence of $\{\psi_n\}\subset C^\infty_c(\R^2_+)$ convergent to $f$ in $L^2(\R^2_+)$ 
and such that $\mathfrak{s}^\circ[\psi_n-\psi_m,\psi_n-\psi_m]\to0$ as $n,m\to\infty$. But
$$
\mathfrak{s}^\circ[\psi_n-\psi_m,\psi_n-\psi_m]=\langle G^\circ_{\a,\b}(\psi_n-\psi_m),\psi_n-\psi_m\rangle =
\sum_{j=1}^2\langle \mathfrak{d}^\circ_j(\psi_n-\psi_m), \mathfrak{d}^\circ_j(\psi_n-\psi_m)\rangle 
$$
and the latter required convergence to 0 follows since $\mathfrak{d}^\circ_j\psi_n$, $j=1,2$, being convergent in $L^2(\R^2_+)$, are Cauchy sequences there.
\end{proof}

Finally, we define the minimal and the maximal operators related to $G^\circ_{\a,\b}$. These two operators are important because self-adjoint extensions of 
$G^\circ_{\a,\b}$  lie in between. We follow the well known pattern of construction of these two operators; see e.g. \cite[Sections 1.3.2 and 15.1]{Sch}. 

First, we introduce the following concept. For $f\in L^1_{\rm loc}(\R^2_+)$ we say that the weak  $G^\circ_{\a,\b}$-\textit{derivative} of $f$ exist provided there is $g\in L^1_{\rm loc}(\R^2_+)$ such that
$$
\int_{\R^2_+}G^\circ_{\a,\b}\vp(r,s)\overline{f(r,s)}\,drds=\int_{\R^2_+} \vp(r,s)\overline{g(r,s)}\,drds, \qquad \vp\in C^\infty_c({\R^2_+}).
$$ 
Then we set $G_{\a,\b;{\rm weak}}^{\circ}f:=g$. Next, we define the $G^\circ$-Sobolev space (of order 2)
$$
W_{\a,\b}^{\circ;2}(\R^2_+)=\big\{f\in L^2(\R^2_+)\colon G^\circ_{\a,\b,{\rm weak}}f\,\,{\rm exist}\,\,{\rm and}\,\,{\rm is} \,\,{\rm in}\,\,\,L^2(\R^2_+)\big\}
$$

Define $G^\circ_{\a,\b;{\rm min}}:=\overline{G^\circ_{\a,\b}}$, the closure of $G^\circ_{\a,\b}$, and $G^\circ_{\a,\b;{\rm max}}$, the operator given by 
\begin{equation*}
G^\circ_{\a,\b;{\rm max}}f:=G^\circ_{\a,\b;{\rm weak}}f\quad {\rm for} \quad f\in \D\,G^\circ_{\a,\b;{\rm max}}:=W_{\a,\b}^{\circ;2}(\R^2_+).
\end{equation*}
Notice that  $G^\circ_{\a,\b}$ is closable since $\D\,(G^\circ_{\a,\b})^*$ is dense in $L^2(\R^2_+)$; in addition 
$(G^\circ_{\a,\b})^*=\big(\overline{G^\circ_{\a,\b}}\big)^*$ 

Observe that $G^\circ_{\a,\b;{\rm min}}\subset G^\circ_{\a,\b;{\rm max}}$ because $G^\circ_{\a,\b;{\rm max}}$ is closed. To verify the last claim 
assume that $f,g\in L^2(\R^2_+)$ are such that for some $f_n\in \D\,G^\circ_{\a,\b;{\rm max}}$ we have $f_n\to f$ and $G^\circ_{\a,\b;{\rm weak}}f_n\to g$ 
in $L^2(\R^2_+)$. This means that for every $n$,
$$
\langle G^\circ_{\a,\b}\vp,f_n\rangle=\langle \vp, G^\circ_{\a,\b;{\rm weak}}f_n\rangle,
$$
and passing with $n\to \infty$ gives $\langle G^\circ_{\a,\b}\vp,f\rangle=\langle \vp, g\rangle$, which says that $G^\circ_{\a,\b; {\rm weak}}f$ exists and equals $g$.

\begin{proposition} \label{pro:minmax} 
We have  $(G^\circ_{\a,\b;{\rm min}})^*=G^\circ_{\a,\b;{\rm max}}$ and $(G^\circ_{\a,\b;{\rm max}})^*=G^\circ_{\a,\b;{\rm min}}$. Moreover,  
every self-adjoint extension $S$ of $G^\circ_{\a,\b}$ satisfies  $G^\circ_{\a,\b;{\rm min}}\subset S\subset G^\circ_{\a,\b;{\rm max}}$.
\end{proposition}
\begin{proof} We begin with the first equality and prove the inclusion $(G^\circ_{\a,\b;{\rm min}})^*\supset G^\circ_{\a,\b;{\rm max}}$. Let 
$f\in \D\,G^\circ_{\a,\b;{\rm max}}$. This means that $G^\circ_{\a,\b;{\rm weak}}f$ exists and belongs to $L^2(\R^2_+)$. In other words, it holds
\begin{equation*}
\forall \vp\in C^\infty_c(\R^2_+)\quad \langle  G^\circ_{\a,\b} \vp,f\rangle=\langle \vp,G^\circ_{\a,\b;{\rm weak}}f\rangle,
\end{equation*} 
which means that $f\in \D\,(G^\circ_{\a,\b})^*=\D\,\big(\overline{G^\circ_{\a,\b}}\big)^*$ and $G^\circ_{\a,\b;{\rm weak}}f=(G^\circ_{\a,\b})^*f=(\overline{G^\circ_{\a,\b}})^*f$. 

To prove the opposite inclusion, $\subset$, let $f\in\D\,(G^\circ_{\a,\b;{\rm min}})^*$. This means, in particular, that 
$$
\forall\vp\in C^\infty_c(\R^2_+)\quad \langle G^\circ_{\a,\b}\vp,f\rangle=\langle \vp,(G^\circ_{\a,\b;{\rm min}})^*f\rangle.
$$
But this means that,  $G^\circ_{\a,\b; {\rm weak}}f$ exists and equals $(G^\circ_{\a,\b;{\rm min}})^*f$, so $f\in \D\,G^\circ_{\a,\b;{\rm max}}$.

Since $G^\circ_{\a,\b;{\rm min}}$ is closed, we have $(G^\circ_{\a,\b;{\rm min}})^{**}=G^\circ_{\a,\b;{\rm min}}$ and thus the second equality is a consequence of the first one. The last claim of the proposition is obvious.
\end{proof}
An immediate but important consequence of Proposition~\ref{pro:minmax} is that two self-adjoint extensions of $G^\circ_{\a,\b}$ differ if and only if their domains differ. 

We are now prepared to state and prove the second main result of this section.
\begin{theorem} \label{thm:oper2}
Let $\a,\b>-1$. Then $\mathbb{D}_{\a,\b}^\circ=\mathbb{G}_{\a,\b}^\circ$.
\end{theorem}
\begin{proof} 
We shall prove the inclusion $\mathbb{D}_{\a,\b}^\circ\subset \mathbb{G}_{\a,\b}^\circ$. This is enough due to the well-known maximality property of 
self-adjoint operators: if $S_1$ and $S_2$ are self-adjoint in $(H,\langle\cdot,\cdot\rangle)$ and $S_1\subset S_2$, then $S_1=S_2$, see e.g. \cite[p.\,42]{Sch}. 
In fact, by Proposition~\ref{pro:minmax}, it suffices to check that 
\begin{equation}\label{kkk}
\D\,\mathbb{D}_{\a,\b}^\circ\subset \D\,\mathbb{G}_{\a,\b}^\circ.
\end{equation}
To do this we first reinterpret  $\D\,\mathbb{G}_{\a,\b}^\circ$ defined in terms of the $\calG^\circ$ transform. Recall that $\D\,\mathbb{G}_{\a,\b}^\circ$ 
consists of all $f\in L^2(\R^2_+)$ such than $\Theta_\a \calG_{\a,\b}^\circ f\in L^2(\Gamma^\circ)$, see Section~\ref{sec:self}. We claim that 
$$
\D\,\mathbb{G}_{\a,\b}^\circ=W_{\a,\b}^{\circ;2}(\R^2_+).
$$
Indeed, for $\supset$ take $f\in W_{\a,\b}^{\circ;2}(\R^2_+)$, so that $f\in L^2(\R^2_+)$, $G_{\a,\b;{\rm weak}}^{\circ}f$ exists and is in 
$L^2(\R^2_+)$, i.e., 
$$
\forall\vp\in C^\infty_c(\R^2_+)\quad \langle G_{\a,\b}^\circ \vp, f\rangle = \langle \vp, G_{\a,\b;{\rm weak}}^\circ f\rangle .
$$
Now, the equivalent version of Plancherel's identity \eqref{3.7} from Theorem~\ref{thm:first} implies that also
$$
\langle \calG_{\a,\b}^\circ(G_{\a,\b}^\circ \vp), \calG_{\a,\b}^\circ f\rangle_{L^2(\Gamma^\circ)}
= \langle \calG_{\a,\b}^\circ \vp, \calG_{\a,\b}^\circ (G_{\a,\b;{\rm weak}}^\circ f)\rangle_{L^2(\Gamma^\circ)}.
$$
But, see \eqref{3.2},
$$
\langle \calG_{\a,\b}^\circ(G_{\a,\b}^\circ \vp), \calG_{\a,\b}^\circ f\rangle_{L^2(\Gamma^\circ)}
= \langle \Theta_\a \calG_{\a,\b}^\circ\vp, \calG_{\a,\b}^\circ f\rangle_{L^2(\Gamma^\circ)}
= \langle \calG_{\a,\b}^\circ\vp, \Theta_\a \calG_{\a,\b}^\circ f\rangle_{L^2(\Gamma^\circ)},
$$
hence $\calG_{\a,\b}^\circ (G_{\a,\b;{\rm weak}}^\circ f)=\Theta_\a \calG_{\a,\b}^\circ f\in L^2(\Gamma^\circ)$, so $f\in\D\,\mathbb{G}_{\a,\b}^\circ$. 
For $\subset$, let $f\in \D\,\mathbb{G}_{\a,\b}^\circ$, so that $f\in L^2(\R^2_+)$ and 
$\Theta_\a \calG_{\a,\b}^\circ f\in L^2(\Gamma^\circ)$. Setting $h:= \calG_{\a,\b}^{\circ;-1} (\Theta_\a \calG_{\a,\b}^\circ)f$, so that  
$h\in L^2(\R^2_+)$, it is easily seen, simply by repeating the previous arguments but in the opposite order, that $G_{\a,\b;{\rm weak}}^\circ f$ exists and equals $h$, hence $f\in W_{\a,\b}^{\circ;2}(\R^2_+)$. This finishes proving our claim and thus \eqref{kkk} reduces to 
$$
\D\,\mathbb{D}_{\a,\b}^\circ\subset W_{\a,\b}^{\circ;2}(\R^2_+).
$$ 
But this was essentially done in the proof of Theorem~\ref{thm:oper}. If fact, $f\in \D\,\mathbb{D}_{\a,\b}^\circ$ means that $f\in W_{\a,\b}^\circ(\R^2_+)$ 
and there is $u_f\in L^2(\R^2_+)$ such that for every $g\in W_{\a,\b}^\circ(\R^2_+)$ we have $\mathfrak{t}^{\circ}[f,g]=\langle u_f,g\rangle $. 
In particular, taking for $g$ functions from $C^\infty_c(\R^2_+)$, we obtain 
$$
\forall \vp\in C^\infty_c(\R^2_+)\qquad\int_{\R^2_+} \sum_{j=1}^2\mathfrak{d}_{j}^\circ \vp \overline{\mathfrak{d}_{j,{\rm weak}}^\circ f}
=\langle \vp,u_f\rangle.
$$
But the left-hand side of the above equality equals $\langle G^\circ_{\a,\b}\vp,f\rangle $, see the proof of \eqref{part}, which means that 
$G^\circ_{\a,\b;{\rm weak}}f$ exists and equals $u_f$. In other words, $f\in W_{\a,\b}^{\circ;2}(\R^2_+)$.
\end{proof}

\section{Appendix} \label{sec:app}  

\subsection{The $G_{\a,\b}$ framework} \label{ssec:fra}
In this subsection we comment on explicit forms of notions and objects associated to the $G_{\a,\b}$ framework which correspond to those defined in 
Sections~\ref{sec:Gru}--\ref{sec:self2} in the $G^\circ_{\a,\b}$ setting. Since $G_{\a,\b}=U_{\a,\b}^{-1}\circ G^\circ_{\a,\b}\circ   U_{\a,\b}$, these notions are appropriately related.

A simple computation shows that  the unitary isomorphism 
$$
U_{\a,\b}\colon L^2(d\mu_{\a,\b})\to L^2\big(\R^2_+\big),  \qquad U_{\a,\b}f= r^{\a+1/2}s^{\b+1/2}f,
$$
intertwines the operators $G_{\a,\b}$ and $G_{\a,\b}^\circ$ in the sense that the diagram
\newlength\SRR
 \settowidth\SRR{$\calS(d\mu_{\a,\b})$}
  \newlength\SetaCC
 \settowidth\SetaCC{$\calS_\eta(C_+)$}
 
  \[
 \xymatrix{
\makebox[\SRR][r]{$L^2(d\mu_{\a,\b}) \supset \D\,G_{\a,\b}$}
\ar@{->}[rr]^{\displaystyle G_{\a,\b}} 
\ar@{->}[dd]_{\displaystyle  U_{\a,\b}}
      && \makebox[\SRR][l]{$L^2(d\mu_{\a,\b}   )$}\ar@{->}[dd]^{\displaystyle U_{\a.\b}}    \\ \\
 \makebox[\SetaCC][r]{$L^2\ \big(\R^2_+\big)\supset \D\,G_{\a,\b}^\circ$} \ar@{->}[rr]_{\displaystyle G_{\a,\b}^\circ} && \makebox[\SetaCC][l]{$ L^2\big(\R^2_+\big)$} 
 }
 \]
commutes; recall that 
$$
\D\,G_{\a,\b}=C^\infty_c\big(\R^2_+\big)=\D\,G_{\a,\b}^\circ,
$$
so 
$U_{\a,\b}(\D\,G_{\a,\b})=\D\,G_{\a,\b}^\circ$. 
Thus, the spectral properties of $G_{\a,\b}$ and $G_{\a,\b}^{\circ}$ are the same. Notably, by Proposition~\ref{pro:zero},   if $|\a|\ge 1$, then $G_{\a,\b}^\circ$ is \esa. 

First, the $\calG$-transform is given analogously as in \eqref{3.1}, simply by dropping the `$\circ$' symbol on both sides of \eqref{3.1}. 
Such dropping is tacitly assumed in all places below when we say `analogously/analogous', and, if necessary, we also replace $L^2(\R^2_+)$ 
by $L^2(d\mu_{\a,\b})$; this will be not repeated throughout these comments. The transform $\calG_{\a,\b}$ satisfies the properties analogous to these in 
\eqref{3.2} and \eqref{3.5}, and the action of $\calG_{\a,\b}$ on $L^2(d\mu_{\a,\b})$ is next given analogously to \eqref{3.4a}. Further, 
$L^2(\Gamma_\a):=L^2(\N\times(0,\infty),dn\times d\mu_\a)$ is defined accordingly. With the inverse transform $\calG_{\a,\b}^{-1}$ defined analogously 
to $\calG_{\a,\b}^{\circ;-1}$,  results parallel to Theorem~\ref{thm:first} and Corollary~\ref{cor:first} follow. 

Next, for the self-adjoint extensions of $G_{\a,\b}$, the multiplication operator $\bbM_\a$ and the self-adjoint extension $\mathbb{G}_{\a,\b}$ of 
$G_{\a,\b}$ are defined analogously to $\bbM^\circ_\a$ and $\mathbb{G}^\circ_{\a,\b}$, and the result parallel to Theorem~\ref{thm:second} holds. 
Moreover, the heat semigroup $\{\exp(-t\mathbb{G}_{\a,\b})\}_{t>0}$ for $\mathbb{G}_{\a,\b}$ consists of the integral operators 
\begin{equation*}
\exp(-t\mathbb{G}_{\a,\b})f(r,s)=\calka\calka K_{t;\a,\b}\big((r,s),(u,v)\big)f(u,v)\,d\mu_{\a,\b}(u,v), \qquad f\in L^2(d\mu_{\a,\b}),
\end{equation*}
with kernels 
$$
K_{t;\a,\b}\big((r,s),(u,v)\big)=(ru)^{-\a-1/2}(sv)^{-\b-1/2}K_{t;\a,\b}^\circ\big((r,s),(u,v)\big);
$$
this follows from $\exp(-t\mathbb{G}_{\a,\b})=U_{\a,\b}^{-1}\circ \exp(-t\mathbb{G}^\circ_{\a,\b})\circ U_{\a,\b}$. Explicitly, 
$$
K_{t;\a,\b}\big((r,s),(u,v)\big)=\calka \frac{J_\b(\tau s)}{(\tau s)^\b}\frac{J_\b(\tau v)}{(\tau v)^\b}\exp\Big(-\frac12\frac{\tau(r^2+u^2)}{\tanh 2t\tau}\Big)
\frac{I_\a\big(\frac{\tau ru}{\sinh 2t\tau}\big)}{\big(\frac{\tau ru}{\sinh 2t\tau}\big)^\a }\Big(\frac{\tau}{\sinh 2t\tau}\Big)^{\a+1}\,d\mu_\b(\tau).
$$
Using $J_\b(y)/y^\b\to 2^{\b}\Gamma(\b+1)$ and $I_\a(y)/y^\a\to 2^{\a}\Gamma(\a+1)$ as $y\to 0^+$, for fixed $(r,s)\in\R^2_+$ we can continuously extend 
$K_{t;\a,\b}\big((r,s),(\cdot,\cdot)\big)$ onto the closure $\overline{\R^2_+}$. For instance, for $(u,v)=(0,0)$ we get
$$
K_{t;\a,\b}\big((r,s),(0,0)\big)=2^{\a+\b}\Gamma(\a+1)\Gamma(\b+1)\calka \frac{J_\b(\tau s)}{(\tau s)^\b}\exp\Big(-\frac12\frac{\tau r^2}{\tanh 2t\tau}\Big)
\Big(\frac{\tau}{\sinh 2t\tau}\Big)^{\a+1}\,d\mu_\b(\tau).
$$
Notice also the homogeneity property of the heat kernel expressed by
$$
K_{t;\a,\b}\big((r,s),(u,v)\big)=t^{-(\a+2\b+3)}K_{1;\a,\b}\Big(\Big(\frac r{\sqrt t},\frac st\Big),\Big(\frac u{\sqrt t},\frac vt\Big)\Big),
$$
which is directly seen, but in fact it is a consequence of a homogeneity property of $\mathbb G_{\a,\b}$ inherited from  that for $G_{\a,\b}$, 
cf. Section~\ref{ssec:Grushin}. 

Finally, for self-adjoint extensions of $G_{\a,\b}$ in terms of  sesquilinear forms, the decomposition
$$
G_{\a,\b}=\mathfrak{d}_1^{\dagger} \mathfrak{d}_1+\mathfrak{d}_2^{\dagger} \mathfrak{d}_2
$$
holds with 
$$ 
\mathfrak{d}_1:=\frac{\partial}{\partial r},\qquad \mathfrak{d}_1^{\dagger}:=-\frac{\partial}{\partial r}+\frac{2\a+1}r,
$$
and 
$$ 
\mathfrak{d}_2:=r\frac{\partial}{\partial s},\qquad \mathfrak{d}_2^{\dagger}:=r\Big(-\frac{\partial}{\partial s}+\frac{2\b+1}s\Big),
$$ 
each with $C^\infty_c(\R^2_+)\subset L^2(d\mu_{\a,\b})$ as domain; $\mathfrak{d}_1^{\dagger}$ and $\mathfrak{d}_2^{\dagger}$, are the restrictions of 
the adjoint operators $\mathfrak{d}_1^*$ and $\mathfrak{d}_2^*$ to $C^\infty_c(\R^2_+)$. Consequently, the concept of the weak partial derivatives expressed 
in Definition \ref{def:weak} is analogous and requires, additionally, replacing  $dr\,ds$ by $d\mu_{\a,\b}(r,s)$. Lemma~\ref{lem:weak} then has its analogous 
counterpart, with $-\frac{\a+1/2}r$ replaced by $\frac{2\a+1}r$ and with similar replacement for $-\frac{\b+1/2}s$. Moreover, Definition \ref{def:sob} 
of the delta-Sobolev space has also its analogous counterpart, the operators $\mathbb{D}_{\a,\b}$ and $\mathbb{D}_{0;\a,\b}$ are defined analogously, 
and the analogue of Theorem~\ref{thm:oper} takes place. Also, with an analogous definition of the minimal and maximal operators, the analogue of 
Proposition~\ref{pro:minmax} holds.

\subsection{A reminder on the spectral theory of multiplication operator} \label{ssec:mult}
Let $X$ be a set  with a $\sigma$-algebra $\Xi$ of subsets, and let $\mu$ be a measure on $(X,\Xi)$. 
For a measurable function $\Theta\colon X\to\C$ consider the \textit{multiplication operator} (with `maximal' domain)  on $L^2(X,\mu)$,  
\begin{align*}
\D\, M_\Theta&=\{f\in L^2(X,\mu)\colon \Theta f\in L^2(X,\mu)\},\\
M_\Theta f&=\Theta f,\qquad f\in \D\, M_\Theta.
\end{align*}
Then the following has a straightforward proof.
\begin{lemma} \label{lem:spe}
The multiplication operator $M_\Theta$ defined  above is:

\noindent a) densely defined closed and normal (which means $M_\Theta M_\Theta^*= M_\Theta^* M_\Theta$);

\noindent b) self-adjoint if and only if $\Theta$ is real-valued; 

\noindent c) nonnegative  if and only if $\Theta$ is nonnegative;

\noindent d) bounded  if and only if $\Theta$ is bounded.

Moreover the spectrum $\sigma(M_\Theta)$ coincides with the essential range of $\Theta$.
\end{lemma}
For a real-valued multiplier function we use the $\mathbb{M}$ symbol to emphasize that the corresponding multiplier operator is self-adjoint.
 
One of versions of the spectral theorem states that every self-adjoint operator $S$ on a Hilbert space $H$ is unitarily 
equivalent to a multiplication operator by a real-valued function. More precisely, given such $S$, there is a measure space $(X,\Xi,\mu)$, a real-valued measurable function $\Theta$ on $X$ and a unitary isomorphism $U\colon H\to L^2(X,\mu)$, such that $S=U^{-1}\circ \mathbb{M}_\Theta\circ U$, which in particular means  
$\D\,S=\D(U^{-1}\circ \mathbb{M}_\Theta\circ U)=\{h\in H\colon Uh\in \D\,\mathbb{M}_\Theta\}$.

An important ingredient of the spectral theory of self-adjoint operators on Hilbert spaces is the functional calculus. For a multiplication operator 
it has a simple description. Namely, with notation as above and assuming $\Theta$ to be real-valued, the functional calculus  is understood as the 
mapping $\Phi\to M_{\Phi\circ\Theta}$ assigning to any Borel function $\Phi\colon \sigma(\mathbb{M}_\Theta)\to\C$ the operator 
$\Phi(\mathbb{M}_{\Theta})=M_{\Phi\circ\Theta}$  on $L^2(X,\mu)$, i.e. the multiplication operator by $\Phi\circ\Theta$.

\subsection{Unitary equivalences between the Grushin-type operators} \label{ssec:equi} 
We now point out that the set of type indices $\a,\b\in\R$ in the consideration of the Grushin-type operators can be naturally reduced. 

In the $G_{\a,\b}$ framework, let $V_{\a,\b}$, $\a,\b\in\R$, be the operator of multiplication  by  $r^{-2\alpha}s^{-2\beta}$, i.e. 
$$
V_{\a,\b}f(r,s)=r^{-2\alpha}s^{-2\beta}f(r,s).
$$
Then 
$$
V_{\a,\b}\colon L^2(d\mu_{-\a,-\b})\to L^2(d\mu_{\a,\b})
$$
is a unitary isomorphism. Moreover,
$$
V_{\a,0}\colon L^2(d\mu_{-\a,\b})\to L^2(d\mu_{\a,\b})\quad {\rm and}\quad  V_{0,\b}\colon L^2(d\mu_{\a,-\b})\to L^2(d\mu_{\a,\b}) 
$$
are unitary isomorphisms. Notice that, obviously,  $V_{\a,\b}\big(C^\infty_c(\R^2_+) \big)=C^\infty_c(\R^2_+)$.
\begin{proposition} \label{pro:first}
Let $\a,\b\in\R$. We have
\begin{equation*}
G_{\a,\b}\circ V_{\a,\b}=V_{\a,\b}\circ G_{-\a,-\b}.
\end{equation*}
 Similarly, 
\begin{equation*}
G_{\a,\b}\circ V_{\a,0}=V_{\a,0}\circ G_{-\a,\b},\qquad G_{\a,\b}\circ V_{0,\b}=V_{0,\b}\circ G_{\a,-\b}. 
\end{equation*}
\end{proposition}
\begin{proof} Straighforward calculation.
\end{proof}
From the above proposition it follows that given $\a,\b\in\R$, each of the operators $G_{-\a,-\b}$, $G_{-\a,\b}$ and $G_{\a,-\b}$, is unitarily equivalent 
to $G_{\a,\b}$. Consequently, to perform an analysis of operators $G_{\a,\b}$ for general $\a,\b\in\R$ it suffices to consider only 
$(\a,\b)\in [0,\infty)\times [0,\infty)$. However, as we saw, having the transform $\calG_{\a,\b}$ defined for $\a,\b>-1$, it is reasonable to treat 
$(-1,\infty)\times (-1,\infty)$ as the sufficient range of the type parameters $(\a,\b)$. 

Note that in the $G_{\a,\b}^\circ$ setting the analogue of Proposition \ref{pro:first} reduces merely to the observation that given $\a,\b\in\R$, the 
operators $G_{\a,\b}^\circ$, $G_{-\a,-\b}^\circ, G_{-\a,\b}^\circ$ and $G_{\a,-\b}^\circ$, just coincide.

\subsection{Auxiliary results} \label{ssec:aux}

Firstly, we prove the result which was needed in the proof of Proposition~\ref{pro:zero}. 

\begin{lemma} \label{lem:app1}
Let $\a\in\R$ and $\eta>0$. Then the operator  
$$
L_{\a,\eta}\,\vp(r)=-\vp''(r)-\frac{2\a+1}{r}\vp'(r)+\eta^2r^2\vp(r)
$$ 
with domain $\D\,L_{\a,\eta}= C^\infty_c(0,\infty)\subset L^2(d\mu_\a)$, is essentially self-adjoint if and only if $|\a|\ge1$. Equivalently, the operator 
$$
L_{\a,\eta}^\circ\,\vp(r)=-\vp''(r)+\Big(\frac{\a^2-1/4}{r^2}+\eta^2r^2\Big)\vp(r)
$$  
with $\D\,L_{\a,\eta}^\circ= C^\infty_c(0,\infty)\subset L^2(0,\infty)$,  is essentially self-adjoint if and only if $|\a|\ge1$.  
\end{lemma}
\begin{proof} 
The equivalence follows since the operators $L_{\a,\eta}$ and  $L_{\a,\eta}^\circ$ are unitarily equivalent. It is therefore convenient to consider 
$L_{\a,\eta}^\circ$, which is a Schr\"odinger operator on the half-line $(0,\infty)$ with potential $q_{\a,\eta}(r)=\frac{\a^2-1/4}{r^2}+\eta^2r^2$. 
The spectral theory of such operators is well developed, see, e.g. \cite[Chapter 15]{Sch}. Since $q_{\a,\eta}$ is bounded from below on $(1,\infty)$ and 
$q_{\a,\eta}(r)\ge\frac34r^{-2}$ on $(0,1)$ for $|\a|\ge1$, it follows that for $|\a|\ge1$, $L_{\a,\eta}^\circ$ is essentially self-adjoint; cf. 
\cite[Propositions~15.11-12 and Theorem~15.10]{Sch}. On the other hand, for $|\a|<1$ we have $|q_{\a,\eta}(r)|\le(\frac34-\ve)r^{-2}$ on $(0,1)$ for
some $\ve>0$, hence $L_{\a,\eta}^\circ$ fails to be essentially self-adjoint; cf. again \cite[Propositions~15.11-12 and Theorem~15.10]{Sch}. 
\end{proof}

Secondly, we discuss an alternative way of defining the $\calG^\circ$-transform and its inverse. Namely, assuming $\vp\in C^\infty_c(0,\infty)$ we 
can change in \eqref{3.1} the order of integration to obtain a counterpart to \eqref{3.3} with $\calH_\b^\circ\big[\calL^\circ_{\a,\tau}\vp_s(n)\big](\tau)$ on the left-hand side, where $\vp_s:=\vp(\cdot,s)$. 
This leads to extending this `new' definition onto $L^2\big(\R^2_+\big)$ by setting 
\begin{equation*}
\hat{\calG}_{\a,\b}^\circ f(n,\tau) =\calH_\b^\circ\big[\calL^\circ_{\a,\tau} f_s(n)\big](\tau) , \qquad f\in L^2\big(\R^2_+\big).
\end{equation*}
Comments analogous to those following \eqref{3.4a}, including \eqref{3.4a}, are in order. Accordingly, the \textit{inverse transform} 
$\check{\calG}_{\a,\b}^{\circ;-1}$ on $L^2(\Gamma^\circ)$ is defined by 
$$
\hat{\calG}_{\a,\b}^{\circ;-1}F(r,s)=\sum_{n=0}^\infty (\calH_\b^\circ F_n)(r)\ell^{\a,\circ}_{n,r}(s), \qquad r>0,\quad s>0,
$$
where we let $F_n=F(n,\cdot)$; correctness of above definition for $F\in L^2(\Gamma^\circ)$ is justified similarly as correctness of $\calG_{\a,\b}^{\circ;-1}$. 

The argument for showing that $\calG_{\a,\b}^\circ$ and $\hat{\calG}_{\a,\b}^\circ$ coincide on $L^2(\R^2_+)$, and $\calG_{\a,\b}^{\circ;-1}$ and 
$\hat{\calG}_{\a,\b}^{\circ;-1}$ coincide on $L^2(\Gamma^\circ)$, goes as follows. It is straightforward to check that the analogue of Theorem~\ref{thm:first} holds for $\hat{\calG}_{\a,\b}^\circ$ and $\hat{\calG}_{\a,\b}^{\circ;-1}$ replacing $\calG_{\a,\b}^\circ$ and $\calG_{\a,\b}^{\circ;-1}$. 
Consequently, $\hat{\calG}_{\a,\b}^\circ$ and $\hat{\calG}_{\a,\b}^{\circ;-1}$ are isometries; cf. Corollary~\ref{cor:first}. Now, since $\calG_{\a,\b}^\circ$ and $\hat{\calG}_{\a,\b}^\circ$ coincide on the dense subset $C^\infty_c(0,\infty)\subset L^2(\R^2_+)$, they also coincide on $L^2(\R^2_+)$.

\end{document}